\documentclass[12pt, a4paper]{article}
\usepackage[british]{babel}
\usepackage{amsmath, amsthm, amsfonts, amsbsy, amssymb, scrextend, graphicx}
\usepackage{hyperref} 
\usepackage{verbatim}
\usepackage{color}

\usepackage{tikz}
\usepackage[export]{adjustbox}
\usepackage{float}

\newcommand{\A}{\mathcal{A}}

 \newcommand{\R}{\mathbb{R}}
 \newcommand{\Z}{\mathbb{Z}}
 
\newcommand{\E}{\mathsf{E}}

 \newcommand{\N}{\mathbb{N}}
\renewcommand{\P}{\mathsf{P}}
\newcommand{\Q}{\mathsf{Q}}

 \newcommand{\bx}{{\bf x}}

 \newcommand{\by}{{\bf y}}
 \newcommand{\bz}{{\bf z}}

 \newcommand{\be}{{\bf e}}

\newcommand{\hX}{{\hat X}}

\def\|{{\,|\, }}

\newcommand{\bea}{\begin{eqnarray}}
\newcommand{\eea}{\end{eqnarray}}

\renewcommand{\A}{{\cal A}}

\newcommand{\bean}{\begin{eqnarray*}}
\newcommand{\eean}{\end{eqnarray*}}
\newcommand{\nn}{\nu}
\newcommand{\X}{{\cal X}}

\newcommand{\muNG}{\mu^{(N)}_{\alpha,\beta, G}}

\newcommand{\asm}{{\rm as} ~ m \to \infty}

\newcommand\ga{\alpha}
\newcommand\gb{\beta}

\newcommand\kk{\kappa}
\newcommand\gl{\lambda}

 \newtheorem{theorem}{Theorem}[section]

 \newtheorem{remark}{Remark}[section]
\newtheorem{proposition}{Proposition}[section]

\newtheorem{example}{Example}[section]
\newtheorem{assumption}{Assumption}[section]

\numberwithin{equation}{section}

\theoremstyle{definition}

\begin{document}

\date{}

\title{Probabilistic models motivated by  \\
cooperative sequential adsorption}

\author{
Vadim Shcherbakov
\footnote{Department of Mathematics, Royal Holloway,  University of London, UK. \newline
\indent  Email: vadim.shcherbakov@rhul.ac.uk
}\\
{\small  Royal Holloway,  University of London}
}
\maketitle

\begin{abstract}
{\small 
This survey concerns  probabilistic models motivated by
cooperative sequential adsorption (CSA) models. CSA models are widely used in physics and chemistry for modelling adsorption  processes in which adsorption rates  depend on the spatial configuration  of already adsorbed particles. Corresponding probabilistic
models describe random sequential allocation of particles either 
in a subset of Euclidean space,
or at vertices of a graph. Depending on a technical setup
these probabilistic models are stated 
  in terms of spatial or integer-valued interacting birth-and-death processes. 
In this survey we consider 
several such models that have been studied in recent years.
}
\end{abstract}

\bigskip 


\bigskip 


\tableofcontents

\section{Introduction}

This survey  concerns  probabilistic models motivated by
cooperative sequential adsorption (CSA) models. 
Adsorption is a real life phenomenon which  can be thought of as follows.
Consider particles (e.g. molecules) diffusing around 
a surface of a material. When a particle hits the surface, it can be retained
(adsorbed) by the latter. CSA describes adsorption 
 process in which adsorption rates  depend on the spatial configuration  of existing  particles.  In other words,  particles  adsorb to a surface subject to interaction with previously adsorbed particles. For example,  adsorbed particles
can either attract, or repel subsequent arrivals. 
These types of  interactions 
 are  common for many physical, chemical  and biological processes.

In physics and chemistry cooperative effects in adsorption  are usually studied 
by experiments and by computer  simulations of an appropriate 
CSA model. CSA model is a probabilistic model
 for random sequential deposition of  particles
(e.g. points or objects of various shape)
either in a bounded region of  a continuous space, or at vertices (sites)
of a graph (e.g. lattice).
In such a model a particle is placed at location $x$ with the probability 
that is proportional to a specified function of the 
current configuration of existing particles  in a neighbourhood of $x$.
Such a construction is technically flexible
 for modelling both attractive  and repulsive interaction
 between a new particle and previously adsorbed particles,
 and can be used in modelling CSA like processes in many real-life 
 applications.

The paper is organised as follows. 
In Section~\ref{CSA-continuum} we introduce a model for random sequential deposition of particles (points) in a bounded subset of Euclidean space. 
This continuous model can be naturally interpreted as 
 a model for  time series of spatial locations. Fitting the model 
 to data requires estimation of model parameters. We  show that 
 statistical inference for the model parameters can be based on 
maximum  likelihood estimation. In particular, we describe the corresponding 
estimation procedure and discuss asymptotic properties of maximum likelihood estimators.
In Section~\ref{graph-model} we consider a discrete model 
random sequential deposition of particles at vertices of a graph
(a growth process with  graph based interaction). 
A probabilistic model obtained from the growth process 
by allowing  deposited particles to depart is considered in Section~\ref{revers}.
This model is motivated by adsorption processes
 in which adsorbed particles can be released from the adsorbing substrate. 
 The model is described in terms of a reversible Markov chain and can be 
 regarded as an interacting spin model and closely related to 
 such well known models of statistical physics and  
interacting particle systems as the  contact process and the Ising model.
Finally, in Section~\ref{CSA-point}  we consider  a point process 
motivated by the CSA model. The point process is a probability measure given by a density 
with respect to Poisson point process (in finite dimensional Euclidean space)
and belongs to a class of point processes used in spatial statistics 
for modelling point patterns.

\section{Continuous CSA model}
\label{CSA-continuum}

In this section we consider a probabilistic model for sequential deposition 
of points in a bounded domain of Euclidean space. 
This model is a continuous analogue of  the lattice CSA known as 
 monomer filling with nearest neighbour cooperative effects (see~\cite{Evans1}
 and references therein).
 This lattice model describes a random sequential 
 deposition of particles on the lattice, 
 where  only one particle can be allocated at a site. The probability of allocating a particle 
  at an empty site  is proportional to the allocation rate, which 
  depends on the number of existing particles in a neighbourhood  
 of the site. 
  For example,  the model on the one-dimensional lattice is specified by parameters 
  $c_i,$ $i=1,2,3$.
  Namely, 
  a particle is placed at an empty site $k$ with the rate $c_i$, if the total number 
  of existing  particles at its nearest neighbours $k-1$ and $k+1$ is equal to $i$.
  In the continuous analogue of the lattice CSA model  
  particles (points) are placed sequentially at random into a bounded 
  region of $\R^d$ as follows.
  Given the current configuration of points, the probability of the event 
 that a particle is placed at location $x$ 
 is proportional to  a  parameter $\beta_k\geq 0$ (called the growth rate), where 
 $k$ is the number of existing  points within a given distance $R$ of $x$.
 This  continuous CSA  model was proposed in~\cite{SV1}
 and was further studied as a model for time series of spatial locations in~\cite{Pen-SV1}
 and~\cite{Pen-SV2}.

In the rest of this section we formally define the model and 
discuss statistical inference for the model parameters.

\subsection{The model definition}
\label{CSA-def}

Start with some notations.
Let $\N$ be the set of all positive integers and 
 $\Z_{+}=\N\cup \{0\}$.  Let $\R=(-\infty, \infty)$ and 
$\R_{+}=[0,\infty)$. 
By ${\bf 1}_{A}$ we denote
 the  indicator  function of a set or an event  $A$. 
We assume  that all random variables under consideration are defined on a certain probability 
 space with the probability measure $\P$. The expectation with respect to $\P$ will be denoted by $\E$.

 Given points $x,y\in \R^d$ we denote by  $\lVert x-y\rVert$ the Euclidean distance 
 between $x$ and $y$.
Given a positive number $R$ points $x,y\in \R^d$ are called neighbours, 
 if  $\lVert x-y\rVert\leq R$,
 in which case we write $x\sim y$.
Given a finite set (ordered or unordered) 
$\X$ of points in 
$\R^d$,  define 
\begin{equation}
\label{nn}
\nn(x,\X)=\sum_{y\in \X}{\bf 1}_{\{\lVert x-y\rVert\leq R\}}\quad\text{for}\quad
x\in \R^d,
\end{equation}
in other  words, $\nn(x,\X)$ is 
 the number of neighbours of $x$  in the set
$\X$. By definition $\nn(x,\emptyset)=0.$

The continuous CSA model with the interaction radius $R>0$ 
and parameters  $(\beta_k\geq 0,\, k\in\Z_{+})$ is the probabilistic 
model for 
random sequential deposition of points in $\R^d$ defined as follows.  
Consider a compact convex set $D\subset \R^d$ 
(called the target region, or the observation window).
 Let   $X(k)=(X_1,...,X_k),\, k\geq 0,$ be the sequence 
 of locations of first $k$ points allocated in $D$ according to the model.
By definition, $X(0)=\emptyset$.
Given that $X(k)=x(k)=(x_1,...,x_k)$ for $k\geq 0$
the conditional probability density function of the next point $X_{k+1}$ is
\begin{equation}
\label{dens0}
\psi_{k+1}(x|x(k))=\frac{\beta_{\nn(x, x(k))}}{\int_{D}\beta_{\nn(y, x(k))}dy},\,\, 
x\in D.
\end{equation}
The joint density of $(X_1,...,X_{\ell}),\, \ell\geq 1$ is given by 
\begin{equation}
\label{dens}
p_{\ell,\beta,D}
(x_1,\ldots,x_\ell)=\prod_{k=1}^\ell \psi_k(x_k|x(k-1))=\prod_{k=1}^{\ell}\frac{\beta_{\nn(x_k, x(k-1))}}
{\int_{D}\beta_{\nn(u, x(k-1))}du},\quad x_i\in D,\, i=1,\ldots,\ell.
\end{equation}

The described CSA can be regarded as a discrete 
time spatial birth process with birth rates $\beta_{\nn(x,x(k))},\, x\in D$, provided 
that the state of the process at time $k$  is $x(k)=(x_1,\ldots,x_{k})$.

Alternatively, the  model
can be described  as the  acceptance-rejection sampling described below.
Namely, let $(Y_i,\, i\geq 1)$ be a sequence of independent 
random points uniformly distributed in $D$, 
 and construct another sequence 
of random points by accepting  each
 point of the original  sequence
  with a certain probability to  be described below, 
otherwise rejecting that point.  
 Let $X(k)=(X_1,\ldots, X_{k})$ be the sequence
of  $k=k_n$  accepted points from the finite  sequence  $Y_i,\, i=1,\ldots,n$.
By definition $X(0)=\emptyset$.
The point  $Y_{n+1}$  is accepted with probability   
$\beta_{\nn(Y_{n+1}, X(k))}/C,$ where $C$ is an arbitrary constant
such that $\max_{0\leq i \leq k}\beta_i\leq C$. 
Regardless of the particular choice of $C$, 
 the next accepted point $X_{k+1}$
has the probability density
$\psi_{k+1}(x|X(k))$ given by~\eqref{dens0} 
 In other words, given the sequence $X(\ell)$ of the
 first $\ell$ accepted points, 
the next accepted point $X_{\ell+1}$
is sampled  from a distribution  which is specified by 
the 
probability density proportional to the function  $\beta_{\nn(x, X(\ell))},\, x\in D$
(the value of $C$ influences only the number of discarded points $Y_i$
until the next acceptance).

\begin{remark}
{\rm 
The defined above continuous CSA model was introduced in~\cite{SV1}.
 In that paper    the asymptotic structure 
 of the model point pattern was studied under the assumption 
 that the sequence  $(\beta_n>0,\, n\geq 0)$ 
converges to a positive  limit as  $n\to \infty$.
This  assumption can be interpreted as if ``adsorption rates
stabilize at saturation''.  
}
\end{remark}

\begin{remark}
{\rm 
A special case of the model, when $\beta_i=0$ for $i\geq 1$, is called 
random sequential adsorption (RSA) model.
RSA is also  known  as the car parking model.
In the latter cars are modelled by balls of radius $R$. Cars 
sequentially  arrive  to the target region $D$ and choose
a location to park at random.
 A new arrival is discarded with probability $1$, 
if it overlaps any of previously parked cars. 
Otherwise it  is parked (accepted)  with probability $1$.
}
\end{remark}

\subsection{CSA as a model for time series of spatial locations}
\label{TS}

It has been noted by physicists (e.g.~\cite{Evans1}) that CSA 
model can be used for modelling sequential point patterns in disciplines  such as geophysics, biology and  ecology in   situations, where  a  data set is  presented by a 
sequential  point pattern, i.e., a collection of spatial events
which appear sequentially. 
In other words, 
CSA  can be used as an approximation of spatial spread dynamics in various applications.
This idea was explored in~\cite{Pen-SV1} and~\cite{Pen-SV2}, where 
the continuous CSA
was  regarded as a model for time series of spatial locations,
which is  flexible for modelling both regular and clustered  point patterns
 (e.g. see Figure~\ref{fig240}). Note that models for clustered point 
 patterns are of a particular interest in spatial statistics.
\begin{figure}[htbp]
\centering
\begin{tabular}{cc}
    \includegraphics[width=2.0in, height=2.0in, angle=270]{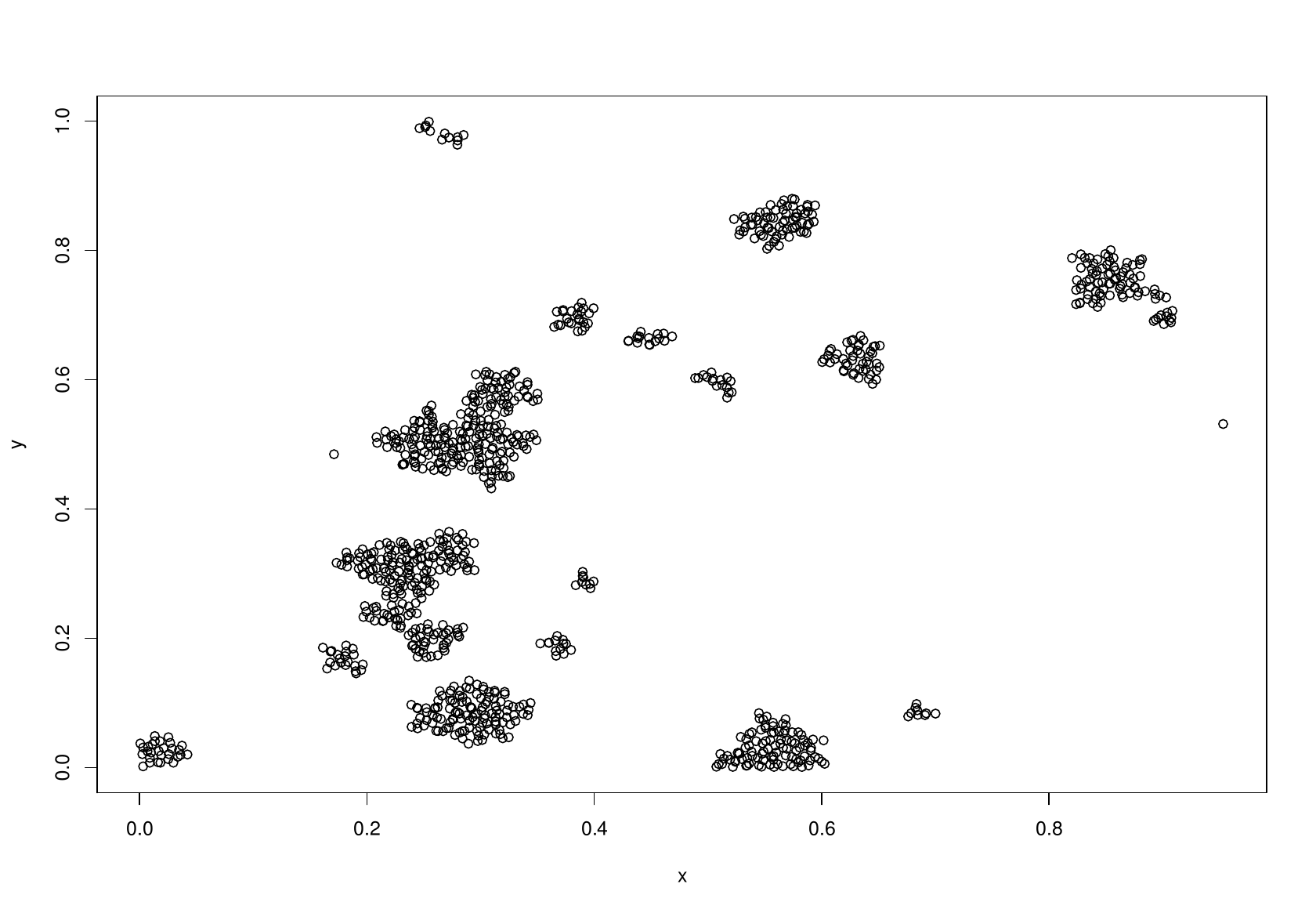}&
  \includegraphics[width=2.0in, height=2.0in, angle=270]{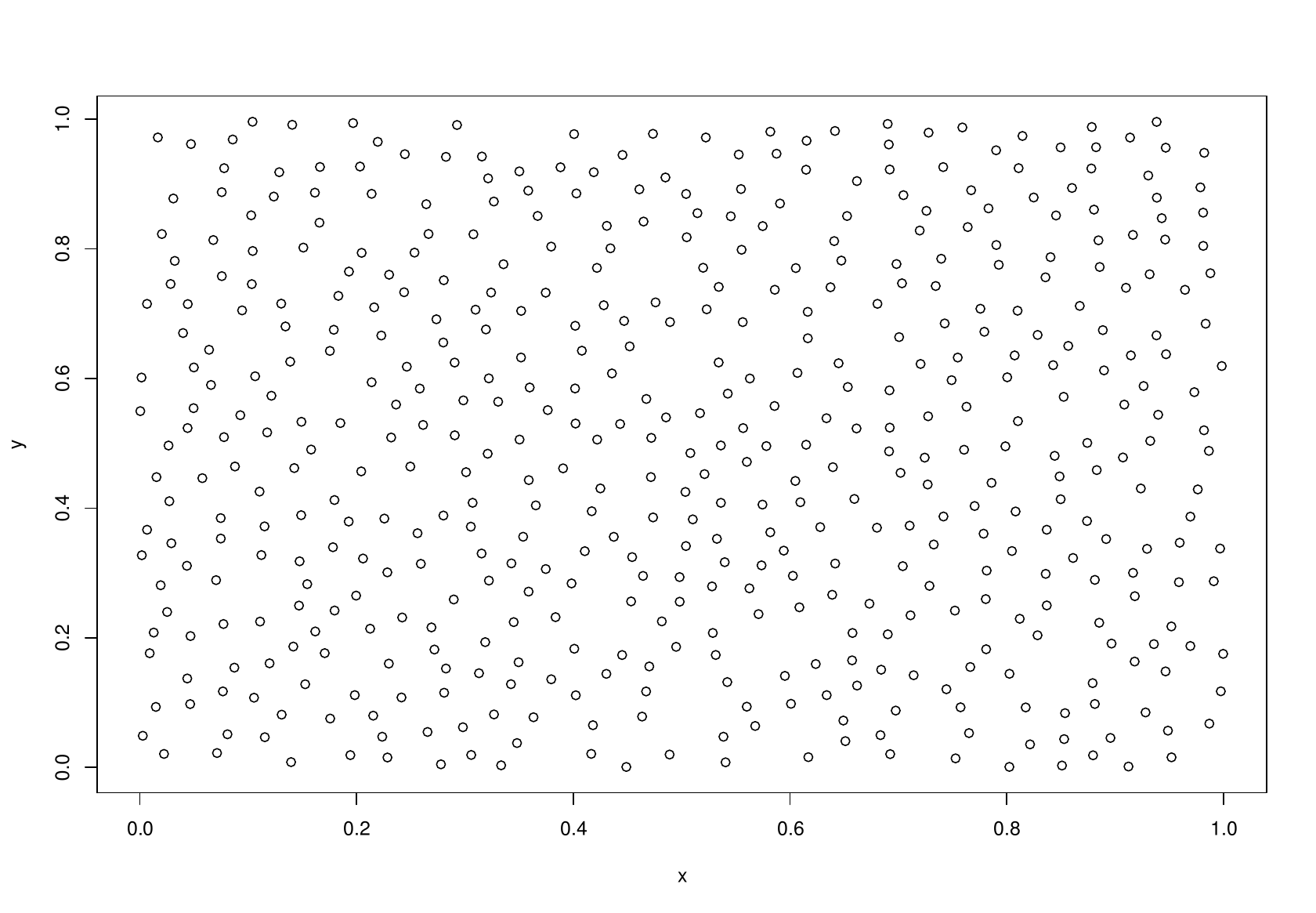}
\end{tabular}
\caption{{\footnotesize CSA simulations in $D=[0,1]^2$.
 Left: $1000$ points, $R=0.01$, $(\beta_i)_{i \geq 0}
= (1,1000,10000,0,0, \ldots)$
Right: $500$ points, $R=0.03$, $(\beta_i)_{i \geq 0}=(1,0,0,\ldots)$,
i.e. this is RSA model.
}}
\label{fig240}
\end{figure}

Fitting a parametric statistical model to real-life data requires 
estimation of the model parameters.
Statistical inference based on maximum likelihood 
estimation was developed   in~\cite{Pen-SV1} and~\cite{Pen-SV2}
 for CSA with a finite number of parameters $\beta$', which means 
 that  there exists a  fixed positive integer  $N$
such that 
\begin{equation}
\label{NN}
\beta_k >0\quad\text{for}\quad 0 \leq k \leq N\quad\text{and}\quad
\beta_k=0\quad\text{for}\quad k\geq N+1.
\end{equation}
It is easy to see  that
the density (\ref{dens0})
(and hence (\ref{dens})) is unaffected by multiplication
of all parameters $\beta_k$ by a constant.
 Therefore, for identifiability of the model we also assume that   $\beta_0=1$, so that the model is parameterised by  parameters $\beta=(\beta_1,\ldots, \beta_N)$.

In general, both number $N$ and  the interaction radius $R$ 
are also regarded as the model parameters and 
have to be estimated. 
The parameter $N$ can be easily estimated by the maximal number 
of neighbours that a point has in an observed pattern (formal definition is given below).
In contrast, estimation of the interaction radius in the general case is 
an open problem. 
If $N=0$, i.e. in the case of RSA model, the natural estimator 
of the interaction radius is the minimal distance between an observed  point 
and those points in the pattern that arrived earlier.
In what follows, we assume that the interaction radius $R$ is assumed to be
 a fixed and  known constant.

\subsection{MLE for CSA}
\label{mle-csa}

In this section we explain how to develop 
statistical inference for parameters of the CSA model 
by using the method of  maximum likelihood estimation (MLE).

Start with considering  the model likelihood.
Recall that ${\bf 1}_{A}$ stands for  the indicator of a set $A$.
Given  an observation
$x(\ell)=(x_1,...,x_{\ell})\in D^{\ell}$, where $\ell\geq 2$,
define 
\begin{align}
\label{Gajdef}
\Gamma_{j,k}&=
\int\limits_{D}{\bf 1}_{\{u:\nn(u,x(k))=j\}}du\quad
\text{for}\quad 0\leq j\leq k\leq\ell-1,\,  j=0,...,N,\\
\label{Gajdef1}
\Gamma_{j,k}&=0\quad\text{for}\quad k<j,\, j=1,...,N,\quad\text{and}
\quad \Gamma_{0,0}=|D|.
 \end{align}
Observe  that 
$$\int_{D} \beta_{\nn(u,x(k))} du =
\sum\limits_{j=0}^{k}\beta_{j}\Gamma_{j,k}
= \Gamma_{0,k} + \sum\limits_{j=1}^{N}\beta_{j}\Gamma_{j,k}.
$$
Further, define 
\bea
\label{tkdef}
t_{k,\ell}=t_k(x(\ell)) :=
\sum\limits_{i =1}^\ell
 {\bf 1}_{\{\nn(x_i,x(i-1))=k\}}\quad
 \text{for}\quad k=0,\ldots,N.
\eea
In terms of statistics~\eqref{Gajdef} and~\eqref{tkdef} we have the following 
equation for the model likelihood
\begin{equation}
\label{like}
p_{ \ell, \beta,D}(x_1,\ldots,x_\ell)  =  
\frac{\prod_{j=1}^{N}\beta_{j}^{t_{j}(x(\ell))}}
{\prod_{k=1}^{\ell}\int_{D}\beta_{\nn(u, x(k-1))}du} 
=\frac{\prod_{j=1}^{N}\beta_{j}^{t_{j,\ell}}}
{\prod_{k=1}^{\ell}\left(\Gamma_{0,k-1} +
\sum_{j=1}^{N}\beta_j\Gamma_{j,k-1}\right)}.
\end{equation}
The log likelihood function is therefore given by  
 \begin{equation}
 \begin{split}
 L_D(x(\ell), \beta) : =& \log(p_{\ell, \beta,D}(x_1,\ldots,x_{\ell}))\\
 &=  \sum_{j=1}^{N}t_{j, \ell}\log(\beta_j)
  - \sum\limits_{k=1}^{\ell}\log\left(\Gamma_{0,k-1} +
\sum\limits_{j=1}^{N}\beta_j\Gamma_{j,k-1}\right).
\end{split}
\end{equation}

\begin{remark}
\label{R2}
{\rm 
Note that since $X(0)=\emptyset$
the first point $X_1$ is uniformly distributed in $D$, and, also, 
 the first term  in the sum $\sum_{k=1}^{\ell}...$
in the preceding display is just a constant $\log(|D|)$.
}
\end{remark}

Maximum likelihood estimators (MLEs)  are defined as usual, i.e. 
as maximizers of the model
likelihood and can be found by solving MLE equations obtained 
by equating to zero the log-likelihood derivatives.

If $N$ is unknown, then  it has to be estimated before estimating $\beta$'s.
Given $X(\ell)=(X_1,...X_{\ell})$, where, as before, we assume that $\ell\geq 2$, 
 we estimate $N$ by 
\begin{equation}
\label{Nestim} 
\widehat{N}=\widehat{N}(X(\ell)) :=\max\limits_{X_i\in X(\ell)}\nn(X_i, X(i-1))
= \max \{j:t_{j,\ell} >0\}.
\end{equation}
It is easy to see  $\widehat{N}$ is the maximum likelihood estimator of the  parameter $N$.

Having estimated $N$ by  $\widehat{N}$
we have that 
\begin{eqnarray*}
L_D(X(\ell), \beta) 
  =   \sum_{j=1}^{\widehat{N}}t_{j, \ell}\log(\beta_j)
  - \sum\limits_{k=2}^{\ell}\log\left(\Gamma_{0,k-1} +
\sum\limits_{j=1}^{\widehat{N}}\beta_j\Gamma_{j,k-1}\right).
\end{eqnarray*}
The maximum likelihood estimator
$
\widehat{\beta}(X(\ell))=(\widehat{\beta}_{1},\ldots,
\widehat{\beta}_{\widehat{N}},0,0,\ldots)$ 
of the true parameter vector 
$(\beta_1^{\left(0\right)},\ldots,
\beta_{N}^{(0)})$
is defined as  the
maximizer of the log likelihood $L_D(X(\ell), \beta)$ over
vectors of the form $(\beta_1,\ldots,\beta_{\widehat{N}},0,0\ldots)$.
Since 
$L_D(X(\ell), \beta)$ 
depends smoothly on $(\beta_1,\ldots,\beta_{\widehat{N}})$; 
 the 
$(\widehat{\beta}_{1},\ldots,\widehat{\beta}_{\widehat{N}})$  
 is a solution to the  system of MLE equations
\begin{equation}
\label{ml0}
\frac{\partial L_D(X(\ell), \beta)}{\partial \beta_j}=
0,\quad j=1,\ldots,\widehat{N},
\end{equation}
or, equivalently, 
\begin{equation}
\label{ml}
t_{j, \ell} - \sum\limits_{k=2}^{\ell}
 \frac{\beta_j\Gamma_{j, k-1}}{\Gamma_{0,k-1} +
\sum_{i=1}^{\widehat{N}}\beta_i\Gamma_{i,k-1}}=
0,\quad j=1,\ldots,\widehat{N},
\end{equation}
where note that the sum $\sum_{k=2}$ in the right hand  side starts from $k=2$ 
because of Remark~\ref{R2}.
It is obvious that $\widehat{N} \leq N$ almost surely.
If $\widehat{N}< N$,  then 
$t_{j, \ell}=0$ for $N'+1\leq j \leq N$. 
It is also  possible that 
$t_{j, \ell}=0$ for some $j<\widehat{N}$. 
Therefore, 
if an observed point pattern is not a ``typical'' model pattern, then 
we might not have sufficient information to estimate 
the full set of parameters.
However, if $\widehat{N}=N$
and all $t-$statistics 
are positive,  then  there exists a unique  
positive solution $(\widehat{\beta}_{1},\ldots,\widehat{\beta}_{N})$ of the likelihood equations.
It turns out that  these conditions hold 
 with probability 
tending to $1$,
as the amount of observed information increases in a certain  natural sense
(to be explained). 

\begin{example}
\label{exam1}
{\rm  Suppose that  $N=1$, i.e., there is one unknown parameter $\beta=\beta_1$.
Assume that an observed sequence of points $x(\ell)=(x_1,\ldots,x_{\ell}),\, \ell\geq 2$ is
such that  $\widehat{N}=1$ and the statistic $t_{1,\ell}>0$ (the number of points having $1$ neighbour).
There is a a single MLE equation in this case,that is 
\begin{equation}
\label{t1}
t_{1,\ell}- \sum\limits_{k=2}^{\ell}
\frac{\beta\Gamma_{1,k}}{\Gamma_{0,k} + \beta \Gamma_{1, k}}=0.
\end{equation}
If   $0<t_{1,\ell}<\ell-1$, then  
existence and uniqueness of the solution of the MLE equation follows from the fact 
that  the left hand side of equation~\eqref{t1} is a strictly monotonic 
function of $\beta$.
If $t_{1,\ell}=\ell-1$, then this suggests that the observed pattern 
is generated by the model obtained by  setting formally $\widehat{\beta}=\infty$.
In the  corresponding  limit model a new point is allocated with probability one 
in the neighbourhood of existing points subject to the constraint that it 
cannot have more than one neighbour among those points.
}
\end{example}

\begin{remark}
\label{MC-for-CSA}
{\rm 
It should be noted that  the model   log-likelihood, and, hence, MLEs for the CSA model,
 can be effectively computed numerically by the classical Monte-Carlo
 (required to compute $\Gamma-$statistics that are given 
by integrals).
We refer to~\cite{Pen-SV1} for numerical examples.
}
\end{remark}

\subsection{Asymptotic properties  of MLE estimators}
 
In this section we briefly discuss 
asymptotic properties of MLE estimators for CSA
in the situation, when the amount of observed information increases.
In the classic case of i.i.d. observations this limit regime 
means that the number of observations tends to infinity.
The analogue of this in spatial statistics 
is known as the  increasing domain asymptotic framework, which 
means the number of  observed 
points  tends to infinity, as 
 the target region (observation window) 
 expands to the whole space.
 We describe below  this limit regime in relation to 
 CSA model with a finite number 
of non-zero parameters $\beta$, where there are natural restrictions 
on the number of observed points in a given target region.

Let $D_1$ denote   the unit cube centred at the origin
(or any compact convex set $D_1\subset  \R^d$).
Consider a sequence of rescaled domains $D_m=m^{1/d}D_1,\, m\in \Z_{+}$.
Given  $m$, consider the CSA process as the acceptance/rejection 
sampling with target region $D=D_m$.
Denote by $A_m(n)$ the (random) number of points accepted  out
of the first $n$ incoming points. 
If  $N<\infty$, then  no  particle can be placed at 
any location $x$ with more than $N$ existing particles within distance $R$ 
of $x$. 
Therefore, the limit 
$\theta_m=\lim_{n \to \infty} A_m(n)$ exists almost surely, and is
a finite random variable. 
Further, there exists a finite limit 
$\lim_{m\to \infty}\theta_m =:\theta_\infty$,
known as the jamming density (see~\cite{Pen-SV1} for more details of this quantity).

The increasing domain asymptotic framework in the case of the CSA model can be now defined as follows.
\begin{assumption}
\label{A}
The number  $\ell_m$ of observed points in the domain $D_m$
is  asymptotically linear in $m$ with coefficient
 below the jamming density $\theta_\infty=\theta_{\infty}(R, \beta_1,...,\beta_N)$, 
 that is 
$$
\lim_{m \to \infty}  \left( \frac{\ell_m}{m}
\right)=\mu\in (0,\theta_\infty).
$$
\end{assumption}
Note
 that  the above  limit is known as the {\em thermodynamic limit} in
the statistical physics literature.

{\it  Assume in the rest  of the section that Assumption~\ref{A} holds}.
It turns out that under this assumption 
 the log-likelihood derivatives in the case of CSA model
behave asymptotically very similar to those in the i.i.d. case. 
 This fact allows to combine methods of 
the classic  MLE theory  for  i.i.d. observations
 (e.g., see \cite{Lehman}) with the modern theory 
for sums of locally determined functionals (to be explained)
to establish consistency and asymptotic normality of MLE estimators 
for CSA model under assumption~\ref{A}.

Given parameters  $N$ and $\beta=(\beta_1,\ldots,\beta_N)$
consider  the probability measure  $\P_{m, \beta}$  on finite point sequences of
 length $\ell_m$ in $D_m$ specified by the probability density $p_{\ell,\beta, D}$
 with $\ell=\ell_m$ and $D = D_m$.
Denote $\beta^{\left(0\right)}=\displaystyle{
\left(\beta_1^{\left(0\right)},\ldots,\beta_N^{\left(0\right)}\right)}$ the true 
parameter and let $\P_m^{\left(0\right)}:=\P_{m, \beta^{\left(0\right)}}$.
Given observation $X(\ell_m)\in D_m$ define the maximum likelihood estimators
$$
\widehat{\beta}(m)=\widehat{\beta}(X(\ell_m))=(\widehat{\beta}_{1,m},\ldots,\widehat{\beta}_{N,m})
$$
of parameters $\beta^{\left(0\right)}=(\beta_1^{\left(0\right)},\ldots,
\beta_{N}^{\left(0\right)})$  as those values that maximize
the log likelihood function
$L_m(\beta) :=  L_{D_m}(X(\ell_m), \beta),$
as explained in Section~\ref{mle-csa}.

It was shown in~\cite[Corollary 2.1]{Pen-SV1}
that 
$$\P_m^{\left(0\right)}\left\{\frac{p_{\ell_m, \beta^{\left(0\right)}, D_m}
(X_1,\ldots,X_{\ell_m})}
{p_{\ell_m, \beta, D_m}(X_1,\ldots,X_{\ell_m})} > 1 \right\}\to 1, 
~~~ \asm,
$$
for $\beta=(\beta_1,\ldots,\beta_N)\neq    \beta^{\left(0\right)}=
(\beta_1^{\left(0\right)},\ldots,\beta_N^{\left(0\right)})$. 
This  result is analogous to the well known result
for the case of  i.i.d.observations
(e.g., see~\cite[Chapter, Theorem 2.1]{Lehman})
and
justifies why statistical inference for the CSA model can be based 
on MLE.  
Furthermore, it was  shown in~\cite[Theorem 2.2 and Lemma 5.2]{Pen-SV1}
that
with $\P_m^{\left(0\right)}-$probability tending to  $1$,
as $m\to \infty$, 
the estimator $\widehat{N}$ is equal to $N$ 
and 
there exists a unique positive solution
$(\widehat{\beta}_{1,m},\ldots,\widehat{\beta}_{N,m})$
 of the system of MLEs, such that 
$\widehat{\beta}_{i,m}\to\beta_i^{\left(0\right)}$ for $i=1,\ldots,N,$ 
in $\P_m^{\left(0\right)}-$ probability, as $m\to \infty$.

Asymptotic normality of MLE estimator $\widehat{\beta}$ was established 
in~\cite{Pen-SV2}.
Specifically, it was shown that
$\sqrt{m}(\widehat{\beta}(m)-\beta^{\left(0\right)})\to
\xi(\mu)$ in $\P_m^{\left(0\right)}-$distribution, as $m\to \infty,$
 where
$\xi(\mu)$ 
is the Gaussian vector with zero mean and covariance matrix
given by the inverse matrix of 
the model limit information matrix.
The latter is defined
 as  the limit (in $\P_m^{\left(0\right)}-$probability, as $m\to \infty$)
 of the observed information matrix 
$-\frac{1}{m}\left(\frac{\partial^2L_m(\beta)}
{\partial \beta_i \partial \beta_j}\right)_{i,j=1}^N$
evaluated at the true parameter $\beta^{\left(0\right)}$.
A detailed study of the structure of the information matrix can be 
found~\cite{Pen-SV2} to which we refer for further details.

Usefulness of showing  asymptotic normality of a parameter
 estimator  provides asymptotic justification for  creating
 confidence intervals based on the normal distribution,
when a sufficiently large number of points is observed in a 
sufficiently large region relative to the interaction radius $R$
(see~\cite{Pen-SV2}  for examples of creation of  confidence intervals).

\subsection{MLE for CSA and the theory of locally determined functionals}
\label{local}

The asymptotic analysis of MLEs is based on the fact 
that the model statistics have special structure. 
Namely, these statistics are  sums of so called  locally determined
functionals over a finite set of points.
Below we briefly explain the idea.

Start with some definitions.
A set of points $\X \subset \R^d$ is called  locally finite,
if its intersection with any ball of a finite radius 
consists of a finite number of points. 
A locally determined functional with a given range $r>0$
is a measurable real-valued function $\xi(Y, \X)$
defined for all pairs $(Y, \X)$, where  $Y\in \R^d$ and 
$\X \subset \R^d$ is locally finite,
with the property that 
$\xi(Y, \X)$ is determined only by those  points of
$\X$ that are within distance $r$ of $Y$.
A locally determined functional $\xi(Y, \X)$ is translation invariant 
if $\xi(Y, \X)=\xi(Y+a, \X+a)$ for any $a\in \R^d$.
For example, 
the functional 
\begin{equation}
\label{xi}
\xi(x,\X):={\bf 1}_{\{\nn(x,\X)=j\}},
\end{equation}
 where the  quantity $\nn(x,\X)$ is defined by~\eqref{nn}, 
is a bounded, translation-invariant, locally determined functional with the 
range equal to 
the interaction radius $R$.

Given a locally determined functional $\xi$,  
the corresponding additive functional $H^\xi$  on
finite sequences $X(\ell) = (X_1,\ldots,X_{\ell}) \in (\R^d)^\ell$ is defined 
as follows
\begin{equation}
\label{loc_func2}
H^\xi(X(\ell))=\sum\limits_{i=1}^{\ell} \xi(X_i,X(i-1)).
\end{equation}
Observe now that both $\Gamma-$statistics~\eqref{Gajdef}
and $t-$statistics~\eqref{tkdef} are sums of locally determined functionals.
Indeed, in the case of $t-$statistics
 \begin{equation}
\label{t}
t_{j, \ell}=\sum\limits_{i=1}^{\ell}
{\bf 1}_{\{\nn(X_i,X(i-1))=j\}},\quad 0\leq j\leq N,
\end{equation} 
we have that the  statistic $t_{j,\ell}$  is the additive functional 
corresponding to the  locally determined functional~\eqref{xi}.
Representation of $\Gamma-$statistics as sums of locally determined functionals 
is more technically involved and we refer to~\cite{Pen-SV1} for further details.

The general  limit theory developed in~\cite{PenYuk}) for additive functionals~\eqref{loc_func2} 
implies that, under Assumption~\ref{A},
there exist strictly positive and continuous in $\mu$ functions 
$(\rho_j\left(\mu, \beta\right),\,1 \leq j \leq N$ and 
$\gamma_{j}\left(\mu, \beta\right),\, 1 \leq j \leq N$, 
 such that 
 \bea
\frac{t_{j,\ell_m}}{m}\to \rho_j\left(\mu, \beta\right)
\quad\text{and}\quad 
\frac{\Gamma_{j,\ell_m}}{m}\to \gamma_{j}\left(\mu, \beta\right),
\, ~~~ j=0,\ldots,N,
\eea 
in  $\P_m-$probability, as $m\to \infty$, 
and that are related by the system of  equations
\begin{equation*}
\label{t_g}
\rho_j\left(\mu, \beta\right)=\int\limits_{0}^{\mu}
\frac{\beta_j\gamma_{j}\left(\lambda, \beta\right)}{\gamma_{0}\left(\lambda, \beta\right)+
\sum_{i=1}^N\beta_i \gamma_{i}\left(\lambda, \beta\right)}d\lambda,\, 
~~~ j=1,\ldots,N.
\end{equation*}
Further, let 
$\gamma_{j}^{\left(0\right)}(\lambda):=
\gamma_{j}(\lambda, \beta^{\left(0\right)})$ and
$\rho_j^{\left(0\right)}(\mu):=\rho_j(\mu, \beta^{\left(0\right)})$
for $j=1,\ldots,N$. 
Given $\mu \in (0, \theta^{(0)})$, where $\theta^{(0)}=
\theta_\infty(R, \beta_1^{(0)}, \ldots,\beta_N^{(0)})$,
the vector of  true  parameters  
$\beta^{(0)} = (\beta_1^{(0)}, \ldots,\beta_N^{(0)})$ 
 is a solution of the system of equations 
\begin{equation}
\label{mle_lim}
\rho_j^{\left(0\right)}(\mu)-\int\limits_{0}^{\mu}
\frac{\beta_j\gamma_{j}^{\left(0\right)}(\lambda)}
{\gamma_{0}^{\left(0\right)}(\lambda)+
\sum_{i=1}^N\beta_i \gamma_{i}^{\left(0\right)}(\lambda)}d\lambda=0,\, ~~~ 
j=1,\ldots,N,
\end{equation}
which is the infinite-volume limit of the MLE~\eqref{ml}. 

\begin{example}
\label{exam2}
{\rm 
Assume, as in Example~\ref{exam1}, that $N=1$.
In this case the true  single parameter $\beta^{\left(0\right)}=\beta_1^{\left(0\right)}$
is the unique solution of the limit equation 
$$\rho_1^{\left(0\right)}(\mu)-\int\limits_{0}^{\mu}
\frac{\beta\gamma_{1}^{\left(0\right)}(\lambda)}
{\gamma_{0}^{\left(0\right)}(\lambda)+
\beta\gamma_{1}^{\left(0\right)}(\lambda)}d\lambda=0.
$$
Existence and uniqueness of the solution of 
 this equation follows, similarly to the case 
of MLE equation in Example~\ref{exam1}, from 
 monotonicity of the integrand in the parameter $\beta$.

}
\end{example}

\section{CSA growth model on graphs}
\label{graph-model}

In this section we consider a probabilistic model
for random sequential deposition of particles at vertices of a graph.
The model can be regarded as a discrete version 
of the continuous CSA model considered in the previous section.
Recall that in the continuous model 
the probability of the event  that a particle is placed at location $x$ 
 is proportional to  the growth rate $\beta_k\geq 0$, where 
 $k$ is the number of existing  points in the neighbourhood   of $x$.  
In the discrete model we assume that the growth rate at a vertex $v$
is equal to $e^{\alpha k+\beta m}$, where $k$ is the number of existing particles 
at the vertex $v$, $m$ is the total number of existing particles 
in vertices adjacent to $v$, and $\alpha, \beta$ are given constants.
Thus, in general, we distinguish particles at the vertex itself 
and in its neighbours.

This   growth model can be interpreted as an interacting urn model
with a graph based log-linear interaction.
Similarly to urn models, we are interested in the long term behaviour of the growth process.
In particular, we would like to establish in which cases all vertices receive infinitely 
many particles and in which cases all but finitely particles are allocated at a certain subset
of vertices (e.g. at a single vertex).

\subsection{The model definition}
\label{growth-def}

The model set up is as follows.
Consider  an arbitrary finite  graph $G=(V,E)$ with the set of vertices $V$ and the set 
of edge $E$. If vertices $v,u\in V$ are adjacent, 
we call them {\it neighbours} and write $v\sim u$.
If vertices $v$ and $u$ are not adjacent, then we write $v\nsim u$.
By definition, vertex is not a neighbour of itself, i.e. $v\nsim v$.

The growth process with parameters $(\alpha, \beta)\in\R^2$
 on the graph $G=(V, E)$
is a discrete time Markov chain $X(n)=(X_v(n),\, v\in V)\in \Z_{+}^V$ with  the following 
 transition probabilities  
\begin{equation}
\label{trans}
\P(X(n+1)=X(n)+\be_v|X(n)=\bx)=\frac{e^{\alpha x_v+\beta\sum_{u\sim v}x_u}}
{\Gamma(\bx)},\quad \bx=(x_w,\, w\in V)\in \Z_{+}^V,
\end{equation}
where
$\Gamma(\bx)=\sum_{v\in V}e^{\alpha x_v+\beta\sum_{u\sim v}x_u}$
and $\be_v\in\R^V$ is the $v$-th unit vector, i.e. 
 the vector, the  $v-$th coordinate of which 
is equal to $1$,  and all other coordinates are zero.  

\begin{figure}[H]
  \centering
  \includegraphics[scale=0.9,trim={5.2cm 20cm 5cm 5cm},clip]{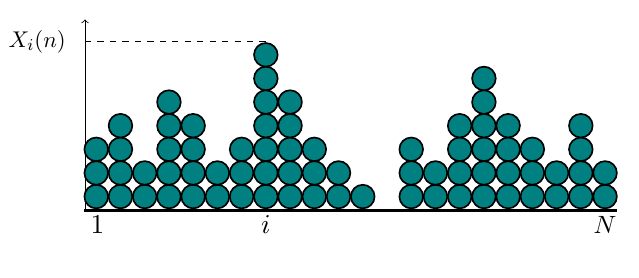}
  \vspace*{-1mm}
  \caption{\small{Deposition model on a linear graph $\{1\sim...\sim N\}$}.}
  \label{fig1}
\end{figure}

The growth process $X(n)=(X_v(n),\, v\in V)$ 
describes a random sequential allocation of particles  on the  graph, where 
  $X_v(n)$ is interpreted as the  number of particles at vertex $v$  at time $n$. 
If  $\beta=0$, then the structure of the underlying graph is irrelevant, and 
the growth  process  is  a special case of the  generalised P\'olya urn 
(GPU) model. 
Recall that GPU model is a model for random sequential allocation of particles at
a finite number of urns, in which  a particle is allocated at an urn $v$ with $x_v$ existing 
particles with probability proportional to the growth rate $f(x_v)$,
 where $f$ is a given positive  function.
The growth process with parameter $\beta=0$ 
is the GPU model with the exponential growth rate $f(k)=e^{\alpha k}$.
If $\beta\neq 0$, then the growth  process 
  can be regarded as an  interacting  urn model obtained 
  by adding graph based interaction. 
Observe that the growth rate 
 $e^{\alpha x_v+\beta\sum_{u\sim v}x_u}$ (i.e. the function  determining the 
 allocation probability~\eqref{trans})  is 
a monotonically increasing  function of the parameter $\beta$.
Therefore, if   $\beta>0$, then  interaction 
between components of the growth process is cooperative  in the sense that 
particles in a neighbourhood of a vertex accelerate the growth rate 
at the vertex.
In contrast, if $\beta<0$, then the interaction between  process's components 
is competitive in the sense that particles in a neighbourhood of a vertex 
 the growth rate slow down the growth at the vertex.

\begin{remark}
{\rm   
The growth process on arbitrary 
graph was introduced in~\cite{MenSV20}.
The growth process  a single parameter  $\lambda:=\alpha=\beta\in \R$
on a cycle   graph 
was introduced and studied in~\cite{SV10a}
(Recall that a cycle graph  with $N\geq 2$ vertices is the 
graph  $G=\{1\sim 2 \sim ... \sim N-1\sim N\sim 1\}$).
 The limit cases of the growth process on a  cycle graph 
   obtained by  setting 
 $\lambda=\infty$ and $\lambda=-\infty$
 (with convention $\infty\cdot 0=0$)  were studied in~\cite{SV10b}
 (see an open problem at the end of Section~\ref{localisation} for more details).
A  version of the growth process on a  cycle graph, where 
the  parameter $\lambda$ depends 
on a vertex (i.e. $\alpha_v=\beta_v=\lambda_v>0$, $v\in V$), 
was studied in~\cite{CMSV}.
}
\end{remark}

\subsection{Localisation in the growth model  with attractive interaction}
\label{localisation}

Recall some known results for  GPU models.
Consider a GPU model  with the growth rate
determined by a function $f$, as described in the preceding  section.
Assume that $f$ is  such that $\sum_{k=1}^{\infty}\frac{1}{f(k)}<\infty$.
It is known that in this case,  
with probability one, all but a finite number of particles are
 allocated at a single random urn.
In other words, the allocation process {\it localises} at a single urn.
This result immediately implies the eventual localisation at a single vertex 
for the growth process with parameters $\alpha>0$ and  $\beta=0$
(as it is just a special case of the aforementioned GPU model).

It was shown in~\cite{MenSV20} that a similar localisation effect 
occurs in  the growth process with attractive 
interaction  introduced  by a positive parameter $\beta$. 
It turns out that in this case  the growth process localises at special subsets
 of  vertices  rather than at a single vertex. 

Recall some definitions from  graph theory necessary to state the result.
  Let  $G=(V, E)$ be a finite graph.
Then, given a subset  of vertices $V'\subseteq V$  the corresponding induced subgraph is a  graph  $G'=(V', E')$ whose edge set $E'$ consists of all of the edges in $E$ that have both endpoints in $V'$.
 The induced subgraph $G'$ is also known as a subgraph induced by  the set of 
 vertices $V'$.
A complete induced subgraph is called a clique, and 
a maximal clique is a clique  that is  not an induced subgraph of another clique. 

The localisation result in~\cite[Theorem 1]{MenSV20} is as follows.
Consider  the growth process $X(n)=(X_v(n),\, v\in V)\in \Z_{+}^V$ 
with parameters $(\alpha, \beta)$ on  a finite connected graph $G=(V, E)$, 
and let $0<\alpha\leq \beta$. Then 
 for every initial state $X(0)\in\Z_{+}^V$ with probability one there exists a random maximal clique  with a vertex set  $U\subseteq V$ such that 
$$\lim_{n\to \infty} X_v(n)=\infty
 \text{ if and only if } v\in U,\, \text{ and }
\lim_{n\to \infty}\frac{X_{v}(n)}{X_{u}(n)}=e^{C_{vu}},\,\text{for} \,v,u\in U,$$
where
$$C_{vu}=
\lambda\lim_{n\to \infty}\sum_{w\in V}X_{w}(n)[{\bf 1}_{\{w\sim v, w\nsim u\}}-
{\bf 1}_{\{w\sim u, w\nsim v\}}],\, \text{if } 0<\lambda:=\alpha=\beta,$$
and $C_{vu}=0$,  if $0<\alpha<\beta$.

The above localisation  effect 
was first shown in~\cite[Theorem 3]{SV10a} 
(see also~\cite[Theorem 1]{CMSV})
in the case  when  $\alpha=\beta>0$  and the underlying  graph is a cyclic graph.
In the special case of the cyclic graph any  clique is just a pair of neighbouring vertices
(assuming that the graph consists of at least three vertices).

The proof  is based on the following key fact.
Namely, given an {\it arbitrary} initial configuration 
the process localises at one of the graph's clique with probability 
that  is {\it bounded away from zero}.
This implies that with probability one  the process eventually localises 
at one of the graph's  cliques (the final clique).

Conditioned that particles are allocated only at vertices of a given clique, 
the numbers of allocated particles at these vertices 
grow according to a  multinomial model in the case when $\alpha=\beta$.
The allocation probabilities of this multinomial model are 
determined by the configuration of existing particles in the neighbourhood 
of the clique which remains unchanged since the start of the localisation.
In other words, the multinomial model is determined by 
the  limit  quantities $C_{vu}$ that  depend on the state of the   process  at the time moment, when localisation starts at the final clique.
In the case $\alpha<\beta$ these quantities irrelevant, and the numbers 
of particles at vertices of the final clique grow   in the same  way 
as in the case of the complete graph described in the example below.

\begin{example}[Complete graph]
{\rm 
Consider the growth process  $X(n)=(X_1(n),...,X_m(n))$ with parameters $0<\alpha<\beta$ 
 on  a complete graph  with $m\geq 2$ vertices labeled by  $1, ...,m$.
 Let 
$Z_i(n)=X_i(n)-X_m(n),\, i=1,..., m-1$.  By~\cite[Lemma 3.3]{MenSV20}
 the process of differences
$Z(n)=(Z_1(n),..., Z_{m-1}(n))\in \Z^{m-1}$ is  an irreducible  positive recurrent Markov chain. 
Positive recurrence 
was shown by applying Foster's criterion (e.g. see~\cite[Theorem 2.6.4]{MPW}) with the 
Lyapunov function given by  
$$
g(\bz)=\sum\limits_{i=1}^{m-1}z_i^2,\, \bz=(z_1,...,z_{m-1})\in \Z^{m-1}.
$$
It should be noted  that exactly  the same fact (see~\cite[]{SV10a})
 is true for the process of differences 
in the GPU model with the growth rate $f(k)=e^{-\lambda k},\, k\in\Z_{+}$, where 
$\lambda>0$.
}
\end{example}

\begin{remark}
{\rm   Localisation also occurs 
 if $0<\beta<\alpha$. In this  case  
 with probability one the growth process 
  localises at a single vertex, which is 
  similar to the GPU model with the growth  rate give by the 
  function  $f(k)=e^{\alpha k}$.
 }
 \end{remark}

\paragraph{Open problem: repulsive interaction}

The long term behaviour of the  growth process in the case 
when $\lambda:=\alpha=\beta<0$
is  largely unknown. In this case 
the interaction between the process's components is repulsive, which 
greatly complicates the study of the process's behaviour.
Below we briefly describe some  known results and state an open problem.

Consider the growth process 
on a cycle  graph $G=\{1\sim 2\sim ... \sim m\sim 1\}$
(i.e. with $m$ vertices labelled by $1,...,m$) with parameters $\lambda:=\alpha=\beta<0$.
Start with a special semi-deterministic case of the process obtained by 
formally setting $\lambda=-\infty$.  Namely, 
in this case given the process's 
state $\bx=(x_1,...,x_m)$ (where $x_i$ denotes, as before, 
the number of particles at vertex $i$) 
a next particle is allocated to a vertex $x_i$ for which the  
quantity $u_i=x_{i-1}+x_i+x_{i+1}$ (with convention $1-1=m$ and $m+1=1$)
is minimal. If there is more than one such vertex, then one of them 
is chosen at random. 
In this semi-deterministic model
if the number of vertices of the graph is $m=4$,
then,   given any initial configuration 
particles will be placed with probability one only at a pair of non-adjacent vertices 
(i.e. either at vertices  $\{1,3\}$, or at vertices $\{2,4\}$). 
Moreover, after a finite number of steps 
the numbers of particles at the vertices of the 
final  pair differ by no more than $1$ and equal to each other 
every other step.
Similar but much more complicated limit behaviour is observed in the case
of the cycle graph with  arbitrary number of vertices. 
We refer to~\cite{SV10b} for further details, where complete classification 
of the long term behaviour of the model with $\lambda=-\infty$ on the cycle graph 
is given. 

In the case $-\infty<\lambda<0$ only a partial result is known in the case of the cycle 
graph with even number of vertices.  Namely, it is shown in~\cite{SV10a}
that if initially there are no particles, then 
with a positive probability particles can be allocated either 
only at even, or only at odd
vertices. The complete classification of the limit behaviour of the model in 
the case $-\infty<\lambda<0$ is an open problem.

\section{The reversible model}
\label{revers}

\subsection{The model definition}

In this section we consider  a probabilistic model which is a
version of the growth process
(defined  in Section~\ref{growth-def}) obtained 
by allowing deposited particles to depart from the graph.

It is convenient to define the  model in terms of a  continuous time
Markov chain (CTMC).
The model set up is as follows.
 Let, as  before, $G=(V,E)$ be a finite graph 
with the set of vertices $V$ and the set of edges $E$.
Recall that $\be_v\in \R^V$  is 
 the vector, the  $v-$th coordinate of which 
is equal to $1$,  and all other coordinates are zero.  
Consider a CTMC $X(t)=(X_v(t),\, v\in V)\in\Z_{+}^V$ and the transition rates 
$r_{\alpha, \beta}(\bx,\by),\, \bx, \by\in \Z_{+}^V$ given by 
\begin{equation}
\label{rates}
r_{\alpha, \beta}(\bx,\by)=\begin{cases}
e^{\alpha x_v+\beta\sum_{u\sim v}x_u},& \mbox{ for }\by=\bx+\be_v \mbox{ and }\bx=(x_v,\, v\in V),\\
1,& 
\mbox{ for } \by=\bx-\be_v \mbox{ and } \bx=(x_v,\, v\in V): x_v>0,\\
0,& \text{otherwise}.
\end{cases}
\end{equation}
where $\alpha, \beta\in\R$  are given constants.

Note that if the death rate was zero, then the CTMC $X(t)$
would be  a continuous time version of the growth 
process.
If~$\beta=0$, then CTMC $X(t)$ is a collection of independent 
 processes labelled by the vertices of graph $G$. In this case 
a component of the Markov chain is a continuous time birth-and-death process
on $\Z_{+}$ (or, equivalently, a nearest neighbour 
random walk on $\Z_{+}$) that evolves as follows.
Given state $k\in\Z_{+}$ it jumps to $k+1$ with the rate $e^{\alpha k}$ and jumps 
to $k-1$ (if $k>0$) with the unit rate. 
Such a process is a special case of the  birth-and-death (BD) process
on the set on non-negative integers. 
The long term behaviour of an integer valued BD process is well known.
 Namely, given a
set of transition characteristics one can, in principle, determine whether the 
corresponding Markov chain MC is (positive) recurrent or (explosive, if the time 
is continuous) transient, and compute various other characteristics of the process.
In particular, the general theory implies 
the following  long term behaviour  of the CTMC $X(t)$ in the independent case
(i.e. when $\beta=0$).
\begin{itemize}
    \item If $\alpha<0$, then each component of $X(t)$ is positive recurrent, and, hence, 
    $X(t)$ is positive recurrent.
    \item If $\alpha=0$, then each component of $X(t)$ is  a reflected symmetric simple random walk on $\Z_{+}$, which is null recurrent. The CTMC $X(t)$ is null recurrent
    if the number of components is either $1$, or $2$, and it is transient
    if the number of components is $3$ or more.
    \item If $\alpha>0$, then each component of $X(t)$ is explosive transient, and, hence, 
    the CTMC $X(t)$ is explosive transient.
\end{itemize}

If $\beta\neq 0$,
then the CTMC $X(t)$ can be regarded as a system of 
  interacting birth-and-death   processes that are 
  labelled by vertices of the graph and evolve
 subject to interaction determined by the parameter $\beta$.
Note that  the presence of interaction can  significantly affect the collective 
behaviour of a system  
and produce effects that might be of interest in modelling the evolution of
multicomponent  random systems.
The  model provides a flexible and mathematically tractable 
choice for  modelling various types of interaction.
For example, if $\beta>0$, then the interaction between 
components is  cooperative meaning that a positive component
accelerates growth of its neighbours. 
In the case  $\beta<0$ the interaction is competitive, since
components suppress growth of each other.

The CTMC $\xi(t)$ was introduced in~\cite{SV15}, where its long term behaviour
was studied in some special cases. In full generality 
the long term behaviour of process was studied in~\cite{Janson}. Main results 
and research methods of these works are  explained below.

\subsection{Long term behaviour of the model}
\label{N=infty}


In this section we review the main results of~\cite{Janson} 
concerning the long term behaviour of the  countable CTMC $X(t)$. 
Recall countable  CTMCs  can be  non-explosive transient, explosive transient, 
null recurrent and positive recurrent. 
It turns out that all these limit behaviours are realised in the case of the CTMC
$X(t)$ depending on parameters $\alpha, \beta$ and on 
the structure of the underlying 
graph. In addition, 
the long term behaviour of the Markov chain is 
largely determined by a relationship between parameters $\alpha, \beta$ 
and the largest eigenvalue of  the graph.

Let us give  some definitions. 
Let $\A=(a_{vu},\, v,u\in V)$ be  the  adjacency matrix of the graph $G=(V,E)$, i.e. $\A$ is a symmetric matrix  such that $a_{vu}=a_{uv}=0$ for $u\nsim v$ and  $a_{vu}=a_{uv}=1$ for $u\sim v$.
Since $\A$ is symmetric, its eigenvalues are real. Denote
them by $\lambda_1(G)\geq \lambda_2(G)\geq\dots\geq  \lambda_n(G)$, so that
$\lambda_1:=\lambda_1(G)$ is the largest eigenvalue (Perron-Frobenius eigenvalue). 
It is well known that   $\gl_1(G)>0$ (except the  case when the graph has no edges).
Not that in terms of the adjacency matrix $\A$ the birth rate in~\eqref{rates1} can be written as follows 
\begin{equation}
\label{b-rate}
e^{\alpha x_v+\beta\sum\limits_{u\sim v}x_u}=e^{\alpha x_v+\beta(\A\bx)_v}.
\end{equation}
Further, an \emph{independent set} of vertices in a graph $G$ is a set of the
vertices such that  no two vertices in the set are adjacent.
The \emph{independence number} $\kappa= \kappa(G)$ of a graph $G$
is  the cardinality of the largest independent  set of vertices.

 Theorem~\ref{main0} below is an extract of~\cite[Theorem 2.3]{Janson}
 that distinguishes between recurrence and transience.
\begin{theorem} 
\label{main0}
Suppose that the graph $G$ is connected 
and $\beta\neq 0$.
\begin{enumerate}
\item The CTMC $X(t)$ is recurrent in the following two cases
\begin{enumerate}
\item   $\ga<0$ and $\ga+\gb\gl_1(G)<0$;
\item   $\ga=0$, $\gb<0$ and $\kk(G)\le2$.
\end{enumerate}
\item 
The CTMC $X(t)$ is transient in all the cases  below
\begin{enumerate}
\item  $\ga>0$; 
\item $\ga=0$ and  $\gb>0$;
\item $\ga=0$, $\gb<0$ and $\kk(G)\ge3$;
\item $\ga<0$ and $\ga+\gb\gl_1(G)\ge0$.
\end{enumerate}
\end{enumerate}
\end{theorem}
The proof of the above result is  greatly facilitated by the fact that the 
CTMC $X(t)$ is reversible, which in turn, allows to apply 
the method of electric networks. This is explained in the next section.

\subsection{Reversibility of the model}
\label{rev}

Let ${\bf I}$ be the unit  $V\times V$ matrix, let $\be\in\R^V$ be the vector all components of which are equal to $1$, 
 and  let $(\cdot, \cdot)$ denote the Euclidean scalar product.
Define functions
\begin{align}
\label{Q}
Q(\bx)&=-\frac{1}{2}((\alpha{\bf I}+\beta\A)\bx,\bx)=
-\frac{\alpha }{2}\sum_{v}x_v^2-\beta \sum\limits_{v\sim u}x_vx_u,\quad
\bx=(x_v,\, v\in V)\in\R^V\\
 \label{S}
S(\bx)&=(\bx, \be)=
    \sum_vx_v, \quad \bx=(x_v,\, v\in V)\in\R^V,\\
    \label{U}
W(\bx)&=-Q(\bx)-\frac{\alpha}{2}S(\bx),
 \quad \bx=(x_v,\, v\in V)\in\R^V.
\end{align}
We claim that the CTMC $X(t)$ is reversible with 
respect to the invariant measure
\begin{equation}
\label{inv-meas}
\begin{split}
e^{W(\bx)}&=e^{-Q(\bx)-\frac{\alpha}{2}S(\bx)}=
e^{\frac{\alpha }{2}\sum\limits_{u}x_u(x_u-1)+
\beta \sum\limits_{w\sim u}x_wx_u}, \, \bx\in \Z_{+}^V,
\end{split}
\end{equation}
Indeed, given $v\in V$ and $\bx\in \Z_{+}^V$ we have  that 
\begin{align*}
-Q(\bx)-\frac{\alpha}{2}S(\bx)
+\alpha x_v+\beta(\A\bx)_v&=-Q(\bx+\be_v)-\frac{\alpha}{2}
S(\bx+\be_v).
\end{align*}
Therefore, 
\begin{align*}
e^{W(\bx)}e^{\alpha x_v+\beta\sum\limits_{u\sim v}x_u}&=
e^{-Q(\bx)-\frac{\alpha}{2}S(\bx)}e^{\alpha x_v+\beta(\A\bx)_v}
 =e^{-Q(\bx+\be_v)-\frac{\alpha}{2}S(\bx+\be_v)}
=e^{W\left(\bx+\be_v\right)}
\end{align*}
which, recalling~\eqref{rates}, means that the detailed balance equation 
\begin{equation}
\label{balance1}
e^{W(\bx)}r_{\alpha,\beta}(\bx,\by)=e^{W\left(\by\right)}
r_{\alpha, \beta}(\by,\bx)
\quad\text{for}\quad \bx, \by\in \Z_{+}^V,
\end{equation}
holds for CTMC $X(t)$ and the invariant measure~\eqref{inv-meas}.

\begin{remark}
{\rm 
It should be noted that 
rewriting equation~\eqref{balance2} as follows
\begin{equation}
\label{balance2}
e^{\alpha x_v}e^{W(\bx)}
=e^{-\beta\sum_{u\sim v}x_u}e^{W\left(\bx+\be_v\right)},
\end{equation}
shows  that the measure $e^{W(\bx)},\, \bx\in \Z_{+}^V$ 
is also invariant for the 
CTMC $Y(t)=(Y_v(t),\, v\in V)$ with the transition rates 
$\hat{r}_{\alpha, \beta}(\bx,\by),\, \bx, \by\in \Z_{+}^V$ given by 
\begin{equation}
\label{rates1}
\hat{r}_{\alpha, \beta}(\bx,\by)=\begin{cases}
e^{\alpha x_v},& \mbox{ for }\by=\bx+\be_v,\mbox{ and } \bx=(x_v,\, v\in V),\\
e^{-\beta\sum_{u\sim v}x_u},& 
\mbox{ for } \by=\bx-\be_v \mbox{ and } \bx=(x_v,\, v\in V): x_v>0,\\
0,& \text{otherwise}.
\end{cases}
\end{equation}
It is shown in~\cite{Janson}  that the long term behaviour of CTMC 
$\widehat{Y}(t)$ is largely the same as the long term behaviour of the CTMC $X(t)$.
}
\end{remark}

Reversibility of a Markov chain 
allows to apply the method of electric networks 
for determining  whether the Markov chain  is recurrent or transient.
The idea is that recurrence/transience 
of the reversible  Markov chain can be established by 
analysing the so called effective resistance of a certain electric
network. 
In the case of the CTMC $X(t)$ the corresponding 
 electric network is defined as follows.

  The  CTMC $X(t)$ can be interpreted as a nearest neighbour 
random walk on the lattice graph $\Z_{+}^V$, i.e. 
the graph with vertices 
$\bx=(x_v,\, v\in V: x_v\in \Z_{+})$, where 
vertices $\bx,\by\in \Z_{+}^V$ are connected by an edge
if the Euclidean distance $\Vert \bx-\by\Vert=1$
(i.e. $\bx$ and $\by$ are nearest neighbours on the lattice).

The electric network on this graph 
is obtained by assigning conductance ($\text{resistance}^{-1}$)
to each edge, which is done as follows.
Given $\bx\in \Z_{+}^V$  and $v\in V$
assign the conductance $e^{W(\bx)}$, where 
the function $U$ is defined in~\eqref{U}  for each
 edge $(\bx-\be_v, \bx)$ with $\bx_v\geq 1$.
 In other words, the edge  conductance 
is equal to the value of the invariant measure of the CTMC $X(t)$ at the state 
 $\bx$. 
The edge $({\bf 0}, \be_v)\, v\in V$, where ${\bf 0}$ is the origin, 
is assigned the unit conductance.
Resistance of an edge is defined the reciprocal of conductance.

By the general method, 
the CTMC $X(t)$ is recurrent (transient), if 
the so called effective resistance of the described electric network 
on the graph $\Z_{+}^V$ is infinite (finite).
The effective resistance in this case if defined, loosely speaking, 
 as the resistance between 
the origin ${\bf 0}$ and  ``infinitite'' vertex (see~\cite{Grimmet} for details).
It turns out that in the case of the CTMC $X(t)$ the effective resistance of the electric network is relatively easy to estimate (see~\cite{Janson} for further details).

Further, the detailed balance equation can be solved 
for the model providing 
analytically tractable equation for the invariant measure. 
This allows to distinguish between 
null and   positive recurrence.

\subsection{Recurrent cases}

The fact that  the invariant measure of the Markov chain is known
allows  to distinguish between the null and positive recurrence
in the recurrent cases of Theorem~\ref{main0}. This is done by analysing 
whether  the invariant measure can be normalised to define 
the stationary distribution.
A direct computation gives that in the case 1(a) of the theorem, i.e. 
when \/ $\ga<0$ and $\ga+\gb\gl_1(G)<0$, the invariant measure
is summable, that is 
\begin{equation}
\label{W-Z}
Z_{\alpha, \beta, G}:=\sum_{\bx\in \Z_{+}^V}
e^{W(\bx)}<\infty.
\end{equation}
Recalling that an irreducible CTMC is positive recurrent if and
only if it has a stationary  distribution and is non-explosive. 
Since a recurrent CTMC is non-explosive we immediately obtain that 
if $\ga<0$ and $\ga+\gb\gl_1(G)<0$, then 
  $X(t)$ is positive
recurrent with the stationary distribution given by
\begin{equation}
\label{measure}
\mu_{\alpha, \beta, G}(\bx)=\frac{1}{Z_{\alpha, \beta, G}}e^{W(\bx)}
\quad\text{for}\quad
 \bx=(x_v,\, v\in V)\in \Z_{+}^V.
\end{equation}
In contrast, the Markov chain is null recurrent in the case 1(b)
of the theorem.
Indeed, if $\alpha=0$, then 
$$Z_{\alpha, \beta, G}=\sum_{\bx\in \Z_{+}^V}
e^{W(\bx)}\geq \sum_{k=0}^{\infty}e^{W(k\be_v)}=\sum_{k=0}^{\infty}1=\infty,$$
where $v\in V$ is any given vertex, i.e. the stationary distribution 
does not exist in this case (regardless of other characteristics of the model).
Therefore, in the case $\alpha=0$ the Markov chain cannot be positive 
recurrent, and, hence, in the recurrent case when 
$\gb<0$ and $\kk(G)\le2$ the Markov chain is  just null recurrent.

\begin{remark}
{\rm In addition, note in~\cite{SV15} positive recurrence of the CTMC $X(t)$
was shown by using the Foster's criterion for positive recurrence in the case when 
$\alpha<0, \beta>0$ and $\alpha+\beta\max_{v\in V} d_v(G)<0$, 
where $d_v(G)$ is the number of neighbours of the vertex $v\in V$, i.e. 
the number of vertices that are  adjacent to $v$ (the degree of the vertex $v$).
The criterion was applied with the Lyapunov function 
$f(\bx)=(Q(\bx),\bx)$, where $Q$ is the quadratic function 
defined in~\eqref{Q}. 
}
\end{remark}

\subsection{Transient cases}

In the case of a transient CTMC it is natural to ask whether 
the CTMC is explosive. 
Note first that  in the transient  case 
2(c)  of Theorem~\ref{main0} 
(i.e. when $\alpha=0, \beta<0$  and $\kappa(G)\geq 3$)
the CTMC $X(t)$ is non-explosive. 
Indeed,  it is easy to see that if  $\alpha=0, \beta<0$, then
the transition rates are uniformly bounded by $1$, and, hence, the process 
cannot be explosive.
In addition, note that in general  there are only two possible long term behaviours of the Markov
 chain if~$\alpha=0$ and~$\beta<0$.
Namely, by the results above 
CTMC $\xi(t)$ is either non-explosive transient or null recurrent, and 
this  depends only on the independence number of the graph~$G$. 

Further, it was shown in~\cite[Lemma 6.3]{Janson}
that in the transient case $\ga<0$  and $\gb=\frac{-\alpha}{\gl_1(G)}$
the CTMC $X(t)$  is not explosive.
Although the transition rates are unbounded in this transient case, the process tends 
to infinity by staying in a domain where the rates are bounded, which prevents the
explosion. This effect is rather easy to understand in the case 
when  the  graph consists  of just two adjacent vertices
(see~\cite[Theorems 1 and 4]{SV15}), when the process is a special 
case of non-homogeneous random walks in the quarter plane.
In this case, if $\alpha<0$ and $\beta=-\alpha$, then the process 
is pushed away from the boundaries (where the rates can be arbitrarily large)  towards the diagonal of the quarter plane, where the  rates are bounded
(see Figure~\ref{Fig3}).  
The same effect of non-explosion takes place in the general case, although 
its  proof  is not straightforward (see~\cite[Lemma 6.3]{Janson}).
It should be noted that the diagonal here is the line determined by the vector $(1,1)^T$, which is the eigenvector of the adjacency matrix $\A=\begin{bmatrix}0&1\\1&0\end{bmatrix}$
of the graph $G=\{1\sim 2\}$, that corresponds to the principle  eigenvalue $1$.
If $\alpha<0$ and $\beta=-\alpha$, then the rates around the diagonal 
are unbounded and the process becomes explosive (see an open problem below concerning
explosion in the general case). 

Further, recall that $d_v(G)$ denotes the number of neighbours
of a vertex $v$.
It was shown in~\cite[Theorems 1 and 2]{SV15} that 
the Markov chain is explosive in the following cases:
\begin{itemize}
\item[(i)] $\alpha>0, \, \beta<0$;
\item[(ii)] $\ga+\gb\min\limits_{v\in V}d_v(G)>0$
including subcases 
\begin{itemize}
\item $\ga>0$ and $\gb\ge0$; 
\item  $\alpha=0$ and $\beta>0$;
\item  $\alpha<0$ and $\beta>\frac{|\alpha|}{\min_{v\in V}d_v(G)}$.
\end{itemize}
\end{itemize}
In particular, in the cases  2(a), 2(b) of Theorem~\ref{main0}
 the CTMC $X(t)$  is explosive.

\paragraph{Open problem: explosion}
Recall that in  general 
$\min_{v\in V}d_v(G) \leq \gl_1(G)$, i.e. the maximal eigenvalue of a graph $G$
is not less than the minimal vertex degree of the graph.
The above results
concerning explosions  do not include the transient  case when 
$$\alpha<0\quad\text{and}\quad 
-\frac{\alpha}{\gl_1(G)}<\beta\leq -\frac{\alpha}{\min_{v\in V}d_v(G)},$$
which remains unsolved.
It was conjectured in~\cite{Janson} that in this case the CTMC $X(t)$ is explosive.
The conjecture is based on the intuition explained in the two-dimensional case above. Namely, that in this transient case the process escapes to infinity by ``following'' 
 the line $\{s{\bf v}_1: s\in R\}$, where ${\bf v}_1$ is the eigenvector corresponding 
 to the largest eigenvalue $\gl_1(G)$, and the transition rates grow exponentially along this 
 line.

\subsection{Phase transition}
\label{Phase}

There is  a phase transition phenomenon 
in the long term behaviour of CTMC $X(t)$ in the case $\alpha<0$.
Indeed, in this case we have the following classification of the process's behaviour 
\begin{itemize}
\item Let $\alpha<0$.
\begin{itemize}
\item[(i)]  If $\beta<-\frac{\alpha}{\gl_1(G)}$, then $X(t)$ is positive recurrent.
This includes the case when  $\beta=0$, 
i.e. when  $X(t)$  is formed
by a collection of independent 
positive recurrent reflected random walks on $\Z_{+}$, and is thus positive recurrent.
\item[(ii)]  If $\beta=-\frac{\alpha}{\gl_1(G)}$, then $X(t)$ is non-explosive 
transient.
\item[(iii)] If $-\frac{\alpha}{\gl_1(G)}<\beta<-\frac{\alpha}{\min_{v\in V}d_v(G)}$, then 
$X(t)$ is transient. It is conjectured that $X(t)$ is explosive transient (see the open problem in the preceding section).
\item[(iv)] If $\beta>-\frac{\alpha}{\min_{v\in V}d_v(G)}$, then 
$X(t)$ is explosive transient.
\end{itemize}
\end{itemize}
If $\beta<0$, then interaction in this case is competitive, as 
neighbours obstruct the growth of each other. The competition implies 
positive recurrence of the process
(which is positive recurrent even without interaction).
Suppose now that  $\beta$ is positive. One could intuitively expect that if $\beta$ is not large, i.e. the cooperative interaction is not strong, so that 
the Markov chain is still positive recurrent.
On the other hand, if $\beta>0$ is sufficiently large,  then  the intuition suggests that 
the Markov chain might become transient.
It turns out, that $\beta_{cr}=\frac{|\alpha|}{\lambda_1(G)}$ is the critical 
value at which the phase transition occurs.
Precisely at this value of $\beta$ the Markov chain is non-explosive 
transient. Moreover, it is conjectured that given $\alpha<0$ the corresponding  critical 
value $\beta_{cr}$ is the only value 
of the parameter $\beta$ when the Markov chain is non-explosive transient. 

\subsection{Examples}

The largest eigenvalue of the adjacent 
matrix of the underlying graph plays essential role 
is determining the long term behaviour of the CTMC $X(t)$.
Estimation of this eigenvalue, and more generally, 
graph eigenvalues, is a very important problem in many applications.
There are well known bounds for  the largest eigenvalue $\gl_1$,
although its explicit  value is known only 
in some special cases.
Below we give several simple examples where 
the largest eigenvalue $\gl_1$ can be computed explicitly,
which allows to rewrite the
conditions of Theorem~\ref{main0} in the case $\ga<0$ in more explicit form.

\begin{example}
{\rm 
 If $G=(V, E)$ is with constant vertex
 degrees $d$, then $\gl_1(G)=\min_{v\in V}d_v(G)=d$.
 For example, $d=1$ for  a graph consisting of two adjacent
 vertices, and $d=2$ in for a cycle graph with at least three vertices.
In this case the Markov chain is positive recurrent if and
 only if $\ga<0$ and $\alpha+\beta d<0$. 
If $\alpha<0$ and 
$\alpha+\beta d=0$, then
the Markov chain is non-explosive transient, and if 
$\alpha+\beta d>0$, then it is explosive transient.
Figures~\ref{Fig3} and~\ref{Fig4} sketch directions of mean jumps of the
process in the simplest case of just  two interacting components.
}
\end{example}

\begin{figure}[htbp]
\centering
\begin{tabular}{cc}
   \resizebox{0.4\textwidth}{!}{\includegraphics{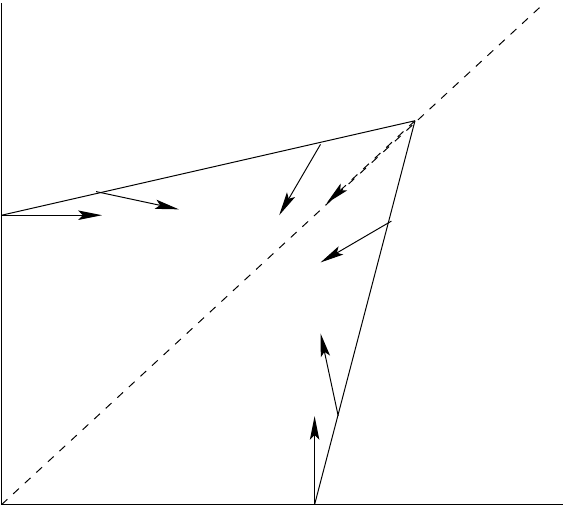}}&
     \resizebox{0.4\textwidth}{!}{\includegraphics{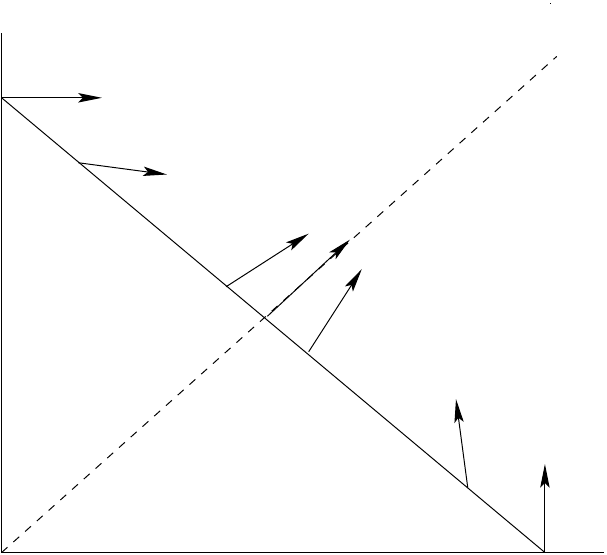}} \\
\end{tabular}
\caption{{\footnotesize $G=\{1\sim 2\}$,
 $\alpha<0$,\,  $\beta>0$.
Left:  $\alpha+\beta<0$;
Right:  $\alpha+\beta\geq 0$.
}}
\label{Fig3}
\end{figure}

\begin{figure}[ht]
\begin{center}
\begin{tabular}{cc}
\resizebox{0.4\textwidth}{!}{\includegraphics{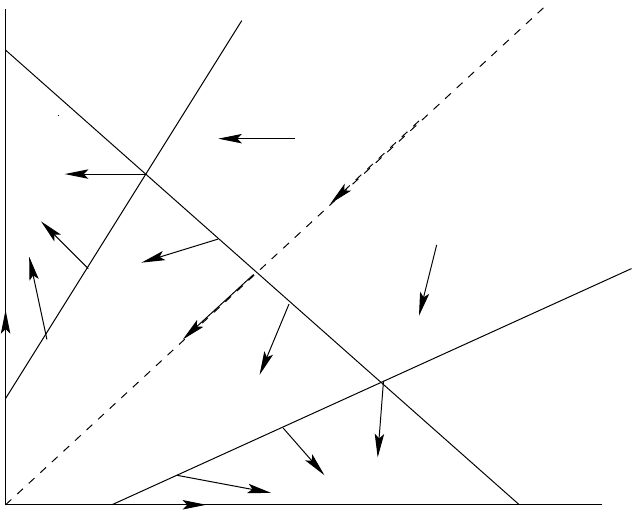}}&
     \resizebox{0.4\textwidth}{!}{\includegraphics{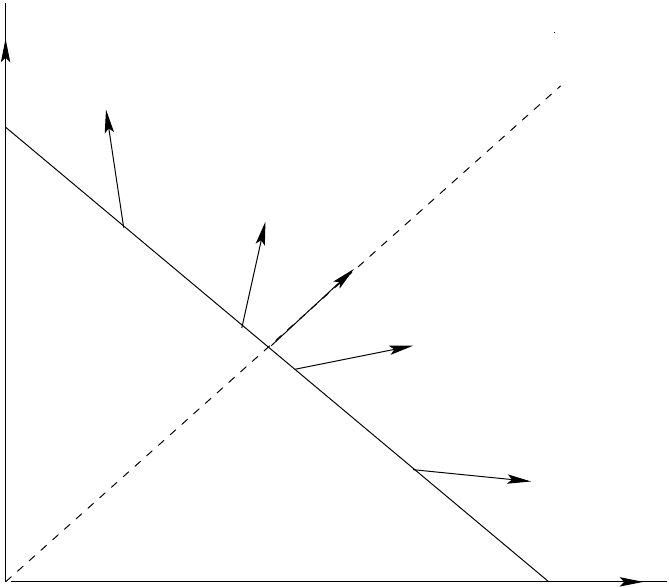}} \\
\end{tabular}
\caption{{\footnotesize $G=\{1\sim 2\}$,  $\alpha>0$,\,  $\beta<0$.
Left:  $\alpha+\beta<0$;
Right:  $\alpha+\beta>0$.
}}
\label{Fig4}
\end{center}
\end{figure}

\begin{example}\label{EK1m}
{\rm 
 Assume that the graph  $G$ is  a star $\mathsf{K}_{1,m}$ with $m=n-1$ non-central
 vertices, where $m\geq 1$. 
A direct computation gives that $\lambda_1=\sqrt{m}$.
Hence, 
the Markov chain is positive recurrent if and
 only if $\ga<0$ and $\alpha+\beta\sqrt{m}<0$. 
If $\alpha<0$ and 
$\alpha+\beta\sqrt{m}\geq 0$, then
the Markov chain is transient.
}
\end{example}

\begin{example}
{\rm 
Consider a linear graph with  $n+2$ vertices, where $n\in \Z_{+}$, that is 
 a graph whose 
 vertices can be enumerated by natural numbers $1,\ldots, n+2$, and such that 
$1\sim 2 \sim \cdots \sim n+1\sim n+2$. 
If $n=0$, then this is the simplest case of  a constant degree graph,
and if $n=1$, then this is the simplest case of  a star graph. 
If $n\geq 2$, then the adjacency 
matrix $\A$ of this graph is the  tridiagonal  matrix given below
$$
\A=\begin{bmatrix}
0 & 1 &  & & & & O  \\
1 & 0 & 1 &  & &  \\
 & 1& \cdot & \cdot & & & \\
& & \cdot & \cdot & \cdot &  & \\
& & & \cdot & \cdot & \cdot &  \\
& & & & \cdot & \cdot & 1\\
O& & & & & 1 & 0
\end{bmatrix}_{(n+2)\times (n+2)}
$$
This is the   tridiagonal symmetric Toeplitz matrix which 
eigenvalues are given by
$$\lambda_k=2\cos\left(\frac{k\pi}{n+3}\right),\, k=1,\ldots, n+2.$$
The maximal eigenvalue is
 $\lambda_1=2\cos\left(\frac{\pi}{n+3}\right)$.
 Thus, the CTMC $X(t)$  is positive recurrent if 
and only  if $\alpha<0$ and $\alpha+2\beta\cos\left(\frac{\pi}{n+3}\right)<0$.
If $\alpha<0$ and $\alpha+2\beta\cos\left(\frac{\pi}{n+3}\right)\geq 0$,  then 
the Markov chain is transient.
}
\end{example}

Two next examples (from~\cite{Janson}) are the cases when the process is null recurrent,
and this essentially is determined by 
the independence number $\kk(G)$.

\begin{example}\label{EK1m0}
{\rm 
 Let, as in Example \ref{EK1m},  $G$ be  a star $\mathsf{K}_{1,m}$, where $m\geq 1$. 
Then $\kk(G)=m=n-1$.
Assume that $\ga=0$ and $\gb<0$.
Then,
the Markov chain is null recurrent if
$n\le3$,
and transient if $n\ge4$.
}
\end{example}

\begin{example}\label{EC0}
{\rm 
 Let  $G$ be  a cycle $\mathsf{C}_{n}$, where $n\geq 3$. 
Then $\kk(G)=\lfloor n/2\rfloor$.
Assume that $\ga=0$ and $\gb<0$.
Then,
the Markov chain is null recurrent if
$n\le5$,
and transient if $n\ge6$.
}
\end{example}

\subsection{The model with finite components}
\label{finite-N}

In this section we consider a version of the  reversible model obtained by 
limiting the maximum number of particles at a vertex. 

As before, let $G=(V,E)$ be a finite connected graph
with the adjacency matrix $\A$.
Let  $\Lambda_{N}=\{0,...,N\}^V$, where $N\geq 1$ is a given natural number.
Consider  a CTMC $\hat X(t)=(\hat X_v(t),\, v\in V)\in\Lambda_{N}$
 with transition rates 
 $\hat r_{\alpha, \beta}(\bx,\by),\, \bx, \by\in \Z_{+}^V$ given by 
\begin{equation}
\label{ratesN}
\hat r_{\alpha, \beta}(\bx,\by)=\begin{cases}
e^{\alpha x_v+\beta\sum_{u\sim v}x_u},& \mbox{ for }\by=\bx+\be_v \mbox{ and }\bx=(x_v,\, v\in V): x_v<N,\\
1,& 
\mbox{ for } \by=\bx-\be_v \mbox{ and } \bx=(x_v,\, v\in V): x_v>0,\\
0,& \text{otherwise}.
\end{cases}
\end{equation}
where $\alpha, \beta\in\R$  are given constants.
In other words, the CTMC $\hX(t)$ evolves precisely  as the CTMC $X(t)$ 
transition rates~\eqref{rates} subject 
to the constraint that at most  $N$ particles can be placed  at a vertex.

Similarly to the CTMC $X(t)$,
 the finite CTMC $\hX(t)$ is irreducible and reversible with the stationary distribution 
 $\mu^{(N)}_{\alpha, \beta, N}(\bx),\, \bx \in \Lambda_N$,
given by 
\begin{equation}
\label{measureN}
\mu^{(N)}_{\alpha, \beta, G}(\bx)=\frac{1}{Z_{\alpha, \beta, G, N}}e^{W(\bx)}
\quad\text{for}\quad
 \bx=(x_v,\, v\in V)\in \Lambda_N^V,
\end{equation}
where the function $U$ is defined in~\eqref{U},
and 
$$Z_{\alpha, \beta, G, N}=\sum\limits_{\bx\in  \Lambda_N^V}
e^{W(\bx)}.$$
By the ergodic theorem for finite irreducible  CTMC's the distribution of 
$\hX(t)$ converges to the stationary distribution~\eqref{measureN},
as $t\to \infty$.

\begin{remark}
{\rm Note that if $N=1$ (in which case the parameter $\alpha$ is redundant) the measure~\eqref{measureN} 
is equivalent to a special case of 
 the celebrated  Ising model on the  graph $G$. Indeed,  the change of 
variables $y_v=2x_v-1$ induces  a  probability measure on $\{-1, 1\}^V$ that is proportional to 
$\exp\left(\frac{\beta}{4}\left(\sum_{v\sim u}y_vy_u+2\sum_{v}y_v\right)\right)$. The latter 
corresponds to  the Ising model with the inverse temperature $\beta/4$ and 
the external field  $h=\beta/2$ on the  graph $G$.  
 }
 \end{remark}

In the rest of this section we 
 show  that, under certain assumptions,  
the probability measure~\eqref{measureN} 
possesses  monotonicity properties
that are similar to those of the ferromagnetic Ising model. 

We start with recalling  some necessary definitions by adopting 
those from~\cite{GHM}.
Let $G=(V,E)$ be an arbitrary graph (not necessarily finite), and let, 
as before,
 $\Lambda_{N}=\{0,..., N\}^{V}$. 
Let ${\cal F}_{G, N}$ be  a standard $\sigma$-algebra  of subsets of $\Lambda_{N}$  
generated by cylinder sets
(if the graph $G=(V,E)$ is finite, then ${\cal F}_{G, N}$ is just a set of all subsets of $\Lambda_{N}$).
Define a  partial order on the set  $\Lambda_{N}$.
  Given  $\bx=(x_v,\, v\in V)\in \Lambda_{N}$ and $\bx'=(x_v',\, v\in V)\in \Lambda_{N}$
we write $\bx\leq \bx'$, if $x_v\leq x_v'$ for all $v\in G$.
A probability measure $\mu$ on $(\Lambda_{N}, {\cal F}_{G, N})$ is said to be monotone
if 
\begin{equation}
\label{mon}
\mu(x_v\geq k|\bx=\bz\text{ off } v)\leq \mu(x_v\geq k|\bx=\by\text{ off } v),
\end{equation}
for all $v\in V$, $k\in \{0,...,N\}$ and $\bz, \by\in \Lambda_{V\setminus\{v\}, N}$  such that $\bz\leq \by$, 
$\mu(\bx=\bz\text{ off } v)>0$ and $\mu(\bx=\by\text{ off } v)>0$.

\begin{theorem}
\label{T1}
Let  $\beta>0$.
Then the probability measure $\muNG$   is monotone.
\end{theorem}
\begin{proof}[Proof of Theorem~\ref{T1}]
Start with an auxiliary   statement (which generalises  Lemma 3.1 in~\cite{Yamb}).
\begin{proposition}
\label{L2}
Let $\P=(p_k, k\in \Z_{+})$ and $\Q=(q_k,\, k\in \Z_{+})$ be discrete 
probability measures on $\Z_{+}$.
If $p_iq_j\leq p_jq_i$ for all $0\leq j < i$,
then   $\P(\{i: i\geq k\})\leq \Q(\{i: i\geq k\})$ for $k\geq 1$, i.e. 
the measure $\Q$ stochastically dominates the measure $\P$.
\end{proposition}
\begin{proof}
A direct computation gives  that 
\begin{align*}
\P(\{i: i\geq k\})-\Q(\{i: i\geq k\})&=\sum_{i=k}^{\infty}p_i-\sum_{i=k}^{\infty}q_i
\pm\left(\sum_{i=k}^{\infty}p_i\right)\left(\sum_{i=k}^{\infty}q_i\right)\\
&=\sum_{i=k}^{\infty}p_i \sum_{j=0}^{k-1}q_j-\sum_{i=k}^{\infty}q_i\sum_{j=0}^{k-1}p_j
=\sum_{ j=0}^{k-1}\sum_{i= k}^{\infty}(p_iq_j-p_jq_i)\leq 0,
\end{align*}
for $k\geq 1$, as required. 
\end{proof}
Given  a vertex $v\in V$ and  configurations 
$\by, \bz\in \Lambda_{N, V\setminus\{v\}}=\{0,1,...,N\}^{V\setminus\{v\}}$,
such that $\by\leq \bz$,
define probability distributions 
$$\P=(p_k=\muNG(x_v=k|\bx\equiv \by\text{ off } v,\, k=0,...,N)$$
and 
$$\Q=(q_k=\muNG(x_v=k|\bx\equiv \bz\text{ off } v,\, k=0,...,N)$$
and  show that 
\begin{equation}
\label{P<Q}
\P(\{k,N\})\leq \Q(\{k, N\})\quad\text{for}\quad k\in \{0,...,N\}.
\end{equation}
A direct computation gives that
\begin{equation}
\label{conditional}
p_k=\frac{e^{\frac{\alpha k(k-1)}{2}+
k\beta(\A\by)_v}}
{\sum_{i=0}^Ne^{\frac{\alpha i(i-1)}{2}+i\beta(\A\by)_v}}\quad\text{and}\quad
\frac{e^{\frac{\alpha k(k-1)}{2}+
k\beta(\A\bz)_v}}
{\sum_{i=0}^Ne^{\frac{\alpha i(i-1)}{2}+i\beta(\A\bz)_v}}.
\end{equation}
Therefore 
\begin{align*}
p_iq_j-p_jq_i&=\frac{e^{\alpha\frac{i(i-1)+j(j-1)}{2}}e^{i\beta(\A\by)_v+j(\A\bz)_v}
\left(1-e^{(i-j)\beta(\A(\bz-\by))_v}\right)}
{\left[\sum_{i=0}^Ne^{\frac{\alpha i(i-1)}{2}+i\beta(\A\by)_v}\right]
\left[\sum_{i=0}^Ne^{\frac{\alpha i(i-1)}{2}+i\beta(\A\bz)_v}\right]}.
\end{align*}
Since $z_u-y_u\geq 0$,  we have that
$$(\A(\bz-\by))_v=\sum_{u\sim v}(z_u-y_u)\geq 0,$$
which gives  
$1-e^{(i-j)\beta(\A(\bz-\by))_v}\leq 0$, and, hence, 
$p_iq_j\leq p_jq_i$ for $0\leq j < i$. 
Equation~\eqref{P<Q} is now 
follows from  Proposition~\ref{L2}. Consequently, 
the measure $\muNG$ is monotone,  as claimed.
\end{proof}

By~\cite[Theorem 4.11]{GHM}, a monotone  probability measure
on $(\Lambda_{N}, {\cal F}_{G, N})$ has positive correlations, that is 
$\muNG(A\cap B)\geq \muNG(A)\muNG(B)$
for any  increasing events $A, B\in {\cal F}_{G, N}$
(an event $A\in {\cal F}_{G, N}$ is said to be increasing if 
${\bf 1}_{\{\bx\in A\}}\leq {\bf 1}_{\{\bx'\in A\}}$ for 
$\bx\leq \bx'$).
It is well known (e.g. see~\cite{Velenik}, ~\cite{GHM})
that positivity of correlations implies
 existence of the  limit for the probability measure~\eqref{measureN}
 in the large graph limit, i.e. 
 as the underlying graph $G$ indefinitely expands in an appropriate sense.
 For example, consider a sequence of graphs $G_n$
 given by  d-dimensional cubes of volume $n$ 
 centered  at the origin. 
If $\beta>0$, then the sequence of corresponding model 
distributions $\mu^{(N)}_{\alpha, \beta, G_n}$ converges to a limit distribution, 
as $n$ tends to infinity
(convergence is understood in the sense of the weak convergence of finite-dimensional distributions).  
This  limit measure  corresponds, in terminology of statistical physics, 
to the so-called empty (zero)  boundary conditions.
Existence of a limit measure in the case of 
other {\it fixed}  boundary conditions (e.g. when 
all spins on the boundary of a graph $G_n$  are equal to $N$) can be shown 
similarly.  
Uniqueness of the limit measure, i.e. that  the limit measure 
does not depend on the boundary conditions,   is an open problem. 

We are now going to show that  the measure $\muNG$ possesses a monotonicity 
property in the parameter $\beta$. Recall that given 
probability  measures $\mu$ and $\mu'$ 
 on $(\Lambda_{N}, {\cal F}_{G, N})$  the measure  $\mu$ is said 
 to be  dominated by $\mu'$ ($\mu\leq \mu'$), 
if  $\mu(A)\leq \mu'(A)$ for every increasing event
$A\in {\cal F}_{G, N}$. 
\begin{theorem}
\label{T1a}
If $\beta_1\leq \beta_2$, then   $\mu^{(N)}_{\alpha,\beta_1, G}\leq
 \mu^{(N)}_{\alpha,\beta_2, G}$.
\end{theorem}
\begin{proof}
Given  $\bx\in\Lambda_{N}$ let  $p_k=\mu_{\alpha,\beta_1, N}(x_v=k|\bx)$   and
$q_k=\mu_{\alpha,\beta_2, N}(x_v=k|\bx)$  for  $k=0,...,N$,
and consider probability distributions 
 $\P=(p_k, k=0,...,N)$ and $\Q=(q_k,\, k=0,...,N)$ on $\{0,...,N\}$.
By the Holley theorem (e.g. Theorem 4.8 in~\cite{GHM}), it suffices 
to show that $\P\leq \Q$, i.e.  $\P(\{k,N\})\leq \Q(\{k, N\})$ for $k=0,...,N-1$.
Using equation~\eqref{conditional}, as in the proof of Theorem~\ref{T1},
we obtain that 
$p_k\sim 
e^{\frac{\alpha k(k-1)}{2}+k\beta_1(\A\bx)_v}$ and 
$q_k\sim e^{\frac{\alpha k(k-1)}{2}+k\beta_2(\A\bx)_v}$, $k=0,...,N$.  
Since $\beta_2-\beta_1\geq 0$ we have that 
$\beta_2(\A\bx)_v-\beta_1(\A\bx)_v=(\beta_2-\beta_1)\sum_{u\in V}x_u\geq 0$, and, hence, 
\begin{align*}
p_iq_j-p_jq_i&\sim e^{\alpha \frac{i(i-1)+j(j-1)}{2}}e^{i\beta_1(\A\bx)_v+j\beta_2(\A\bx)_v}
\left(1-e^{(i-j)\left[\beta_2(\A\bx)_v-\beta_1(\A\bx)_v\right]}\right)\leq 0, \, \text{ if } \, 0\leq j<i.
\end{align*}
By Proposition~\ref{L2}, $\P\leq \Q$, and the theorem follows.
\end{proof}

\section{CSA point process}
\label{CSA-point}

In this section we consider a point process  motivated by the CSA model.
We call this process by the CSA point process. 
The construction of the  CSA point process is reminiscent 
to the CSA time series model in Section~\ref{CSA-continuum}.
The key  difference  between the two models is that the CSA point process 
is a model for {\it unordered} point patterns, while 
the  CSA time series model is a  model for  {\it sequential} point patterns.

The CSA point process is defined, similarly to other 
point processes, as  a probability measure on the set of finite 
 point configurations of a subset of Euclidean space.
 Such a  measure is usually
 specified  by a density with respect  to the Poisson point process with the unit intensity.

Start with some notations and definitions.
Let  $D$ be a compact convex subset of  $\R^d$
 that has a positive Lebesgue measure.
For  $n\geq 1$ define   $n-$point configuration in $D$  as unordered 
set of points $\bx=\{x_1, \ldots, x_n\}$, $x_i\in  D,\, i=1,\ldots, n$, such that 
 $x_i\neq x_j$ $i\neq j$.
Let  $F$ be a set of all finite point configurations in $D$ including 
the empty set $\emptyset$ (the empty set corresponds to $n=0$).
 Let  ${\cal F}$ be a $\sigma-$ algebra of subsets of $F$,
 such that all maps   $\bx \to |\bx\cap B|$, where  $B\subseteq D$ and $\|\cdot\|$ is 
 the cardinality  of a discrete set, are measurable 
 with respect to ${\cal F}$.

Let  $\P_{\Pi}$ be the distribution of Poisson point process  with the unit intensity
on the set $D$, i.e. 
it is the probability measure on $(F,{\cal F})$ given by 
\begin{equation}
\label{d1}
\P_{\Pi}(A)=e^{-|D|} \Big( {\bf 1}_{\{\emptyset \in A\}} + \sum\limits_{n=1}^{\infty}\frac{1}{n!}
\int\limits_{D^n}1_{\{\{x_1,\ldots,x_n\}\in A\}} dx_1\ldots dx_n \Big),\,\,
A\in {\cal F},
\end{equation}
where ${\bf 1}_B$ denotes the indicator function of set $B$.

The CSA point process with the interaction radius $R>0$ and parameters 
$(\beta_m\geq 0,\,m\in\Z_{+})$
is a probability measure on 
$(F,{\cal F})$ specified by the following density (with respect to measure~\eqref{d1})
\begin{equation}
\label{den-csa}
f(\bx)=Z^{-1}\prod_{x_k\in \bx}
\beta_{\nu(x_k, \bx)},
\end{equation}
where $\nu(x_k,\bx)$ is the number of neighbours of a point $x_k$ in a finite point 
configuration $\bx=\{x_1,...,\}$ (defined in~\eqref{nn}),
  \begin{equation}
\label{stat}
Z=e^{-|D|} \Big(1+\sum\limits_{n=1}^{\infty}\frac{1}{n!}
\int\limits_{D^n}\prod_{k=1}^{n}
\beta_{\nu(x_k, \bx)}dx_1\ldots dx_n \Big),
\end{equation}
i.e. $Z$ is the normalising constant.

The process is well defined if $Z<\infty$. 
It was shown in~\cite[Lemma 1]{GrabVS} that, 
if there exists a constant $C>0$,  such that 
 $\beta_m\leq C m^{\alpha}$ for all $m$ with some $\alpha<1$, then 
the CSA point process~\eqref{den-csa} is well-defined.

It turns out that the CSA point process is a special case of the class of 
 interacting neighbours (INP) point process (introduced in~\cite{GrabSar}).
 An INP process is specified 
by a density (with respect to the Poisson point process with the unit intensity)
proportional to a function of the form 
$
\prod\limits_{x_k\in \bx} g(x_k, \bx)$,
where, in turn, 
$g: D\times F\rightarrow \R_{+}$ is a non-negative measurable function.
The density~\eqref{den-csa}
is obtained by setting 
$$
 g(x, \bz)= \sum_{m \geq 0} \beta_m 1_{\{|\bz|=m\}}.
$$

It is easy to see that the construction of the CSA point  process is reminiscent  to the 
 construction of  CSA model for  time series of spatial locations in 
 Section~\ref{CSA-continuum}.  By this similarity
  the CSA point process is also   useful for modelling 
 a wide spectrum of  point configurations from regular ones  to various clustered  point
 patterns.  
 
 Some special cases of the  CSA  point process
 are  the well known in spatial statistics point processes. 
 Consider  several  examples.
\begin{enumerate}
\item  A Poisson point process in the domain $D$ with the intensity $\beta>0$
is obtained for $\beta_i\equiv \beta,\, i\geq 0$.
\item Assume that $\beta_0=\beta$ and $\beta_i=0$ for $i\geq 1$.
The corresponding process is the well known process with hard core interaction
of intensity $\beta$ and the interaction radius $R$.
Realisations of a hard  processes are point patterns, in which 
 the distance between any two points is not less than the interaction 
radius $R$, i.e. a point has no neighbours. 
\item A natural generalisation of the process with the hard core interaction 
is the CSA  process  with a finite number of non zero parameters, that is 
 $\beta_i>0$ for $0\leq i\leq N$ and $\beta_i=0$ for $i>N$, where $N\geq 0$ 
 is a given integer  (if $N=0$, then we  obtain the 
 process with the hard core interaction). 
 Realisations of such a process  are point patterns, in which 
 a point can have no more than $N$ neighbours.
 \item The famous in spatial statistics 
 Strauss point process with the parameters $\alpha>0$ and $0<\gamma<1$
 is obtained by setting $\beta_i=\alpha\gamma^{i/2}$,\, $i\geq 0$.
Traditionally, its distribution is 
specified by a density (with respect to  Poisson point process with the unit intensity) proportional to the function $\alpha^{|\bx|}\gamma^{s(\bx)}$, where $s(\bx)$
is the number of pairs of neighbours in the 
configuration $\bx$. It is easy  to see that $s(\bx)=1/2\sum_{x_k\in \bx}
\nu(x_k, \bx)$, i.e. the density~\eqref{den-csa} is 
 the density of the  Strauss process for the indicated choice of the parameters. 
\end{enumerate}

Consider the CSA point process  with a finite number 
of non zero  $\beta-$parameters, i.e. the process in the item 3 above.
Parameters of this process can be estimated by adopting the 
estimation procedure 
described in the case of the CSA time series model in Section~\ref{mle-csa}.
Namely, assume, as in Section~\ref{mle-csa}  that the interaction radius $R$ is known
(or, somehow estimated).
Then, given an observation $\bx=\{x_1,...,x_n\}$
the parameter $N$ (the number 
of non-zero $\beta-$parameters) can be estimated by 
$\widehat{N}=\max_{x\in \bx}\nu(x, \bx)$.
Non-zero parameters $\beta$ can be estimated by using MLE. However, unlike 
the CSA time series model, the computation 
of MLE estimators here is not so straightforward. The difficulty is that 
 the computation of the  model likelihood in the case of the CSA point process
requires the computation of the normalising constant~\eqref{stat}
which is not analytically tractable.
This is the well-known common problem in MLE estimation 
of parameters of a point process given by a density 
with respect to the Poisson point process.
The normalising constant can be computed/estimated  numerically 
by using the Markov chain Monte-Carlo method. Implementation of the latter 
is not straightforward for point processes and requires advanced simulation techniques 
(e.g. the method of perfect simulation).
This is in contrast to the case of the CSA time series model, where the classic 
 Monte-Carlo is rather effective (see Remark~\ref{MC-for-CSA}).

\end{document}